\documentclass[12pt,a4paper]{article}
\usepackage{amsmath}
\usepackage{amssymb}
\usepackage{amsthm}
\usepackage{amsfonts}
\usepackage{latexsym}

\theoremstyle{plain}
\newtheorem{thm}{Theorem}[section]
\newtheorem{cor}{Corollary}[section]
\newtheorem{lem}{Lemma}[section]
\newtheorem{prop}{Proposition}[section]
\theoremstyle{remark}

\theoremstyle{definition}

\title{Szeg\"{o} kernels and asymptotic expansions\\ for Legendre polynomials}
\author{Roberto Paoletti\footnote{\noindent{\bf Address:}
Dipartimento di Matematica e Applicazioni, Universit\`a degli Studi
di Milano Bicocca, Via R. Cozzi 55, 20125 Milano,
Italy; {\bf e-mail}: roberto.paoletti@unimib.it }}
\date{}

\begin{document}
\maketitle

\begin{abstract}
We present a geometric approach to the asymptotics of the Legendre polynomials
$P_{k,n+1}$, based on the Szeg\"{o} kernel of the Fermat quadric hypersurface,
and leading to complete asymptotic expansions holding on expanding subintervals of $[-1,1]$.
\end{abstract}

\section{Introduction}
The goal of this paper is to develop a geometric approach to the asymptotics of the Legendre polynomial 
$P_{k,n+1}(t)$ for $k\rightarrow +\infty$, with $t=\cos(\vartheta)\in [-1,1]$ and $n\ge 1$ fixed;
as is well-known, $P_{k,n+1}(t)$ is the restriction to $S^n$ of the Legendre harmonic, expressed in polar coordinates
on the sphere. We follow here the terminology of \cite{m0}, \cite{m1} and \cite{ah}.

There is a tight relation between $P_{k,n+1}(t)$ and the orthogonal
projector $$\mathcal{P}_{k,n}:L^2(S^n)\rightarrow V_{k,n},$$ 
where $V_{k,n}$ is the space of level-$k$ spherical
harmonics on $S^{n}$; equivalently, $V_{k,n}$ is the eigenspace of the (positive) Laplace-Beltrami
operator on functions on $S^n$, corresponding to its $k$-th eigenvalue 
$\lambda_{k,n}=k\,(k+n-1)$. 

Namely, for any choice of an orthonormal basis $\big(\varrho_{knj}\big)_{j=1}^{N_{k,n}}$ of $V_{k,n}$
the distributional kernel $\mathcal{P}_{k,n}(\cdot,\cdot)\in \mathcal{C}^\infty\left(S^n\times S^n\right)$ satisfies
\begin{equation}
\label{eqn:projector}
\mathcal{P}_{k,n}\left(\mathbf{q},\mathbf{q}'\right)=\sum_{j=1}^{N_{k,n}}\varrho_{knj}(\mathbf{q})\cdot 
\overline{\varrho_{knj}(\mathbf{q}')},
\end{equation}
where $\mathbf{q}\cdot \mathbf{q}'=\mathbf{q}^t\,\mathbf{q}'$ (we think of $\mathbf{q}$ and
$\mathbf{q}'$ as columns vectors), and $N_{k,n}$ is the dimension of $V_{k,n}$.
By symmetry considerations, $\mathcal{P}_{k,n}\left(\mathbf{q},\mathbf{q}'\right)$
only depends on $\mathbf{q}\cdot \mathbf{q}'$. In fact, with the normalization $P_{k,n+1}(1)=1$,
\begin{equation}
 \label{eqn:projector and legendre}
\mathcal{P}_{k,n}\left(\mathbf{q},\mathbf{q}'\right)=
\frac{N_{k,n}}{\mathrm{vol}(S^n)}\,P_{k,n+1}(\mathbf{q}\cdot \mathbf{q}').
\end{equation}
Thus it equivalent to give asymptotic expansions for $P_{k,n+1}\big(\cos(\vartheta)\big)$
and for $\mathcal{P}_{k,n}\left(\mathbf{q},\mathbf{q}'\right)$ with $\mathbf{q}\cdot \mathbf{q}'=\cos(\vartheta)$.

Since for any $(\mathbf{q},\mathbf{q}')\in S^n\times S^n$ we have
$$
\mathcal{P}_{k,n}\left(\mathbf{q},\mathbf{q}\right)=
\frac{N_{k,n}}{\mathrm{vol}(S^n)},\,\,\,\,\,
\mathcal{P}_{k,n}\left(\mathbf{q},-\mathbf{q}'\right)=
(-1)^k\,\mathcal{P}_{k,n}\left(\mathbf{q},\mathbf{q}'\right),
$$
we may assume $\mathbf{q}\neq \pm \mathbf{q}'$. Then 
there is a unique great circle parametrized by arc length going from $\mathbf{q}$ to $\mathbf{q}'$
in a time $\vartheta\in (0,\pi)$, and $\mathbf{q}^t\,\mathbf{q}'=\cos(\vartheta)$.

Our geometric approach 
uses on the one hand the specific relation 
between spherical harmonics on $S^n$ and the Hardy space of the Fermat quadric hypersurface in $\mathbb{P}^n$ 
(\cite{l}, \cite{g}), and the other hand the off-diagonal scaling asymptotics 
of the level-$k$ Szeg\"{o} kernel of polarized projective manifold (\cite{bsz}, \cite{sz}).

The following asymptotic expansions involve a sequence of constants $C_{k,n}>0$ with a precise geometric meaning
\cite{g}. 
There is a natural conformally unitary
isomorphism between the level-$k$ Szeg\"{o} kernel of the Fermat quadric $F_n\subset \mathbb{P}^n$ and 
$V_{k,n}$, given by a push-forward operation, and $C_{k,n}$ is the corresponding
conformal factor. 

An asymptotic expansion for $C_{k,n}$ is discussed in \cite{g}, building on the theory of \cite{l};
an alternative derivation is given in Proposition \ref{prop:espansione per Cnk}
(with an explicit computation of the leading order term). 

In the following, the symbol $\sim$ stands for \lq has the same asymptotics as\rq.

\begin{thm}
 \label{thm:main expansion}
There exists smooth functions $A_{nl}$ and $B_{nl}$ ($l=1,2,\ldots$) 
on $[0,\pi]$ such that the following holds. 
Let us fix $C>0$ and $\delta\in [0,1/6)$. Then, uniformly in $(\mathbf{q},\mathbf{q}')\in S^n\times S^n$
satisfying $\mathbf{q}^t\,\mathbf{q}'=\cos(\vartheta)$ with
$$
C\,k^{-\delta}<\vartheta<\pi-C\,k^{-\delta},
$$
we have for $k\rightarrow+\infty$ an asymptotic expansion of the form
\begin{eqnarray*}
 \lefteqn{\mathcal{P}_{k,n}(\mathbf{q},\mathbf{q}')}\\
&=& \frac{2^{\frac{n}{2}}}{C_{k,n}^2}\,\left(\frac{1}{\sin(\vartheta)}\right)^{(n-1)/2}\cdot\left[\cos\big(\alpha_{k,n} (\vartheta)\big)\cdot \mathcal{A}_n(\vartheta,k)
+\sin\big(\alpha_{k,n} (\vartheta)\big)\cdot \mathcal{B}_n(\vartheta,k)\right],
\end{eqnarray*}
where 
$$
\alpha_{k,n} (\vartheta)=:k\vartheta+\left(\frac{\vartheta}{2}-\frac{\pi}{4}\right)(n-1),
$$
and
$$
\mathcal{A}_n(\vartheta,k)\sim 1+\sum_{l=1}^{+\infty}k^{-l}\,\frac{A_{nl}(\vartheta)}{\sin(\vartheta)^{6l}},\,\,\,\,\,\,
\mathcal{B}_n(\vartheta,k)\sim \sum_{l=1}^{+\infty}k^{-l}\,\frac{B_{nl}(\vartheta)}{\sin(\vartheta)^{6l}}.
$$
\end{thm}

At the $l$-th step, we have for some constant $C_l>0$
$$
\left|A_{nl}(\vartheta)/\sin(\vartheta)^{6l}\right|,\,
\left|B_{nl}(\vartheta)/\sin(\vartheta)^{6l}\right|
\le C_l\,k^{-l(1-6\delta)},
$$
and a similar estimate holds for the error term.
Hence the previous is an asymptotic expansion for $\delta\in [0,1/6)$.

As mentioned, the same techniques yield an asymptotic expansion for $C_{n,k}$ (see (6.18) in \cite{g}).

\begin{prop}
\label{prop:espansione per Cnk}
For $k\rightarrow +\infty$ we have an asymptotic expansion 
of the form:
\begin{eqnarray*}
C_{k,n}&\sim& 
\left[\frac{(n-1)!}{2\sqrt{2}}\cdot\mathrm{vol}(S^n)\,\mathrm{vol}\left(S^{n-1}\right)\right]^{1/2}
\,(\pi\,k)^{-(n-1)/4}\\
&&\cdot \left[1+\sum_{j\ge 1}k^{-j}\,a_j\right].
\end{eqnarray*}
\end{prop}

If we insert the latter expansion in the one provided by Theorem \ref{thm:main expansion}, we obtain the following:

\begin{cor}
 \label{cor:explicit expansion}
 With the assumptions and notation of Theorem \ref{thm:main expansion}, for $k\rightarrow+\infty$ there is an asymptotic expansion
 \begin{eqnarray*}
 \mathcal{P}_{k,n}(\mathbf{q},\mathbf{q}')&=& \frac{2^{\frac{n+3}{2}}}{(n-1)!}\,
 \frac{1}{\mathrm{vol}\left(S^n\right)\,\mathrm{vol}\left(S^{n-1}\right)}\,\left(\frac{\pi\,k}{\sin(\vartheta)}\right)^{(n-1)/2}\\
&&\cdot\Big[\cos\big(\alpha_{k,n} (\vartheta)\big)\cdot \mathcal{C}_n(\vartheta,k)
+\sin\big(\alpha_{k,n} (\vartheta)\big)\cdot \mathcal{D}_n(\vartheta,k)\Big],
\end{eqnarray*}
where $\mathcal{C}_n(\vartheta,k)$ and $\mathcal{D}_n(\vartheta,k)$ admit asymptotic expansions similar to those of 
$\mathcal{A}_n(\vartheta,k)$ and $\mathcal{B}_n(\vartheta,k)$, respectively (of course, with different functions $C_{nl}$ and 
$D_{nl}$, $l\ge 1$).
\end{cor}

Pairing Corollary \ref{cor:explicit expansion} with (\ref{eqn:projector and legendre}), we obtain:

\begin{cor}
 \label{cor:expansion legendre polynomials}
 In the same situation as in  Theorem \ref{thm:main expansion}, for $k\rightarrow+\infty$ there is an asymptotic expansion
 \begin{eqnarray*}
 P_{k,n+1}\big(\cos(\vartheta)\big)&=& \frac{2^{\frac{n+1}{2}}}{\mathrm{vol}\left(S^{n-1}\right)}\,
 \left(\frac{\pi}{\sin(\vartheta)\,k}\right)^{(n-1)/2}\\
&&\cdot\Big[\cos\big(\alpha_{k,n} (\vartheta)\big)\cdot \mathcal{E}_n(\vartheta,k)
+\sin\big(\alpha_{k,n} (\vartheta)\big)\cdot \mathcal{F}_n(\vartheta,k)\Big],
\end{eqnarray*}
where again $\mathcal{E}_n(\vartheta,k)$ and $\mathcal{F}_n(\vartheta,k)$ admit asymptotic expansions similar to those of 
$\mathcal{A}_n(\vartheta,k)$ and $\mathcal{B}_n(\vartheta,k)$, respectively.
\end{cor}

Let us verify that Corollary \ref{cor:expansion legendre polynomials} fits with the classical asymptotics. 
For example, when $n=1$ we obtain  
$$
P_{k,2}\big(\cos(\vartheta)\big)\sim \cos(k\vartheta)+\cdots,
$$
so that the leading order term is the $k$-th Chebychev polynomial.
Since it is known that in this case the Legendre polynomial \textit{is} the
Chebychev polynomial (\cite{m0}, page 11), this is in fact the only term of the expansion.

For $n=2$, we obtain the formula of Laplace (cfr \cite{leb}, \S 4.6; \cite{o}, (8.01) of Ch. 4; \cite{s}, Theorem 8.21.2), but as
a full asymptotic expansion holding uniformly
on expanding subintervals converging to $[-1,1]$ at a controlled rate, as above:
$$
P_{k,3}\big(\cos(\vartheta)\big)\sim \sqrt{\frac{2}{\pi\,k\,\sin(\vartheta)}}\,
\cos\left(\left(k+\frac{1}{2}\right)\,\vartheta-\frac{\pi}{4}\right)+O\left(k^{-3/2+6\delta}\right).
$$

For arbitrary $n$, $P_{k,n+1}$ is a multiple of a Gegenbauer polynomial (\cite{b}; \cite{m0}, page 16):
\begin{equation}
 \label{eqn:gegenbauer}
P_k^{(n/2-1,n/2-1)}\big(\cos(\vartheta)\big)=r_{k,n}\,P_{k,n+1}\big(\cos(\vartheta)\big).
\end{equation}
Given the standardization for $P_k^{(n/2-1,n/2-1)}$ (\cite{b}, \S 10.8)
\begin{eqnarray*}
r_{k,n}&=&
P_k^{(n/2-1,n/2-1)}(1)={k+n/2-1\choose k}\\
&=&\frac{(n/2)_k}{k!}=\frac{\Gamma(k+n/2)}{k!\,\Gamma(n/2)},
\end{eqnarray*}
where $\Gamma$ is of course the Gamma function. By (35.31) in \cite{m1}, for $k\rightarrow+\infty$ we have
$$
\Gamma(k+n/2)\sim k^{n/2}\,\Gamma(k)=k^{n/2}\,(k-1)!.
$$
Therefore,
$$
r_{k,n}\sim \frac{k^{n/2}\,(k-1)!}{k!\,\Gamma(n/2)}=\frac{k^{n/2-1}}{\Gamma(n/2)}.
$$
If we use the well-known formula (see e.g.  (2) of \cite{m0})
$$
\mathrm{vol}\left(S^{n-1}\right)=\frac{2\,\pi^{n/2}}{\Gamma (n/2)},
$$
we obtain for $P_k^{(n/2-1,n/2-1)}\big(\cos(\vartheta)\big)$ as asymptotic expansion with leading order term
\begin{eqnarray*}
 \lefteqn{2^{\frac{n+1}{2}}\,\frac{k^{n/2-1}}{\Gamma(n/2)}\,\frac{\Gamma (n/2)}{2\,\pi^{n/2}}\,\,
 \left(\frac{\pi}{\sin(\vartheta)\,k}\right)^{(n-1)/2}\,\cos\big(\alpha_{k,n} (\vartheta)\big)}\\
&=&\frac{1}{\sqrt{\pi\,k}}\,\frac{1}{\cos(\vartheta/2)^{(n-1)/2}\sin(\vartheta/2)^{(n-1)/2}}\,\cos\big(\alpha_{k,n} (\vartheta)\big),
\end{eqnarray*}
in agreement with (10) on page 198 of \cite{b}.

\bigskip

\noindent
\textbf{Acknowledgments.} I am endebted to Leonardo Colzani and Stefano Meda for very valuable comments and insights.

\section{Preliminaries}

\subsection{The geometric picture}
For the following, see \cite{g}, \cite{l}.

Let $S^n_1\subset \mathbb{R}^{n+1}$ be the unit sphere, and let us identify the tangent and cotangent
bundles of $S^n_1$ by means of the standard Riemannian metric.
The unit (co)sphere bundles of $S^n_1$ is given by the incidence correspondence
\begin{equation}
 \label{eqn:unit_cosphere}
 S^*\left(S^n_1\right)\cong S\left(S^n_1\right)
 =\left\{(\mathbf{q},\mathbf{p})\in S^n_1\times S^n_1\,:\,\mathbf{q}^t\,\mathbf{p}=0\right\}.
 \end{equation}
 
The Fermat quadric hypersurface in complex projective space is
$$F_n=:\left\{[\mathbf{z}]\in  \mathbb{P}^n\,:\,\mathbf{z}^t\,\mathbf{z}=0\right\};$$ 
let $A$ be the restriction to $F_n$ of the hyperplane line bundle.
Given the standard Hermitian product on $\mathbb{C}^{n+1}$, $A$ is naturally a positive Hermitian line bundle, 
$F_n$ inherits a K\"{a}hler structure $\omega_{F_n}$
(the restriction of the Fubini-Study metric), 
and the spaces of global holomorphic sections of higher powers of $A$, $H^0\left(F_n,A^{\otimes k}\right)$, have an induced
hermitian structure.

The affine cone over $F_n$ is $\mathcal{C}_n=\{\mathbf{z}^t\,\mathbf{z}=0\}\subset \mathbb{C}^{n+1}$; 
the intersection
$
X_1=:\mathcal{C}_n\cap S_1^{2n+1}
$ may be viewed as the unit circle bundle in the dual 
line bundle $A^\vee$.
More generally, for any $r>0$ the intersection 
\begin{equation}
 \label{eqn:circle bundle radius r}
 X_r=:\mathcal{C}_n\cap S^{2n+1}_r
\end{equation}
with the sphere of radius $r$ is naturally identified with the circle bundle of radius $r$ in $A^\vee$.
In particular, 
\begin{eqnarray}
\label{eqn:circle bundle and cosphere bundle}
 X_{\sqrt{2}}=\left\{\mathbf{q}+i\,\mathbf{p}\,:\,\|\mathbf{q}\|^2=\|\mathbf{p}\|^2=1,\,\mathbf{q}^t\,\mathbf{p}=0\right\}
\end{eqnarray}
is diffeomorphic to $S^*(S^n)$ by the map $\beta:(\mathbf{q},\mathbf{p})\mapsto\mathbf{q}+i\,\mathbf{p}$;
furthermore, $\beta$ is equivariant for the natural actions of $O(n+1)$ 
on $S^*(S^n)$ and $X_{\sqrt{2}}$
defined by, respectively, 
$$
B\cdot (\mathbf{q},\mathbf{p})=:(B\mathbf{q},B\mathbf{p}),\,\,\,\,B\cdot (\mathbf{q}+i\,\mathbf{p})=
B\mathbf{q}+i\,B\mathbf{p}\,\,\,\,\,\,\,\,\big(B\in O(n+1)\big).
$$

We shall identify $S^*(S^n)$ and $X_{\sqrt{2}}$, and denote the projection by 
\begin{equation}
 \label{eqn:nu projection}
 \nu:S^*(S^n)\cong X_{\sqrt{2}}\rightarrow S^n,\,\,\,\,\mathbf{q}+i\,\mathbf{p}\mapsto \mathbf{q}.
\end{equation}

There is also a standard structure action of $S^1$ on $X_{\sqrt{2}}$, induced by 
fibrewise scalar multiplication in $A^\vee$, or equivalently in $\mathbb{C}^{n+1}$. 
The latter action
is interwined by $\beta$ with the \lq reverse\rq\, geodesic flow on $S^*(S^n)\cong S(S^n)$. 
The $S^1$-orbits are the fibers of the circle bundle projection
\begin{equation}
 \label{eqn:geodesic flow and circle bundle}
\pi_{\sqrt{2}}:\mathbf{q}+i\,\mathbf{p}\in X_{\sqrt{2}}\mapsto [\mathbf{q}+i\,\mathbf{p}]\in F_n.
\end{equation}

This holds for any $r>0$;
we shall denote by $\pi_r:X_r\rightarrow F_n$ the projection for general $r>0$.


\subsection{The metric on $X_r$}

Let us dwell on the metric aspect of (\ref{eqn:circle bundle radius r}); there are
two natural choices of a Riemannian metric on $X_r$, hence of a Riemannian density, and we need to clarify
the relation between the two.

There is an obvious choice of a Riemannian metric $g_r'$ on  
$X_r$, induced by the standard Euclidean product on
$\mathbb{C}^{n+1}$. With respect to $g_r'$, the $S^1$ orbits on $X_r$ have length
$2\pi\,r$.
Clearly, $g'_r$ is homogeneous of degree $2$ with respect to the dilation
$\mu_r:x\in X\mapsto r\,x\in X_r$, and therefore the corresponding volume form $\Upsilon'_{X_r}$ on
$X_r$ is homogeneous of degree $\dim (X)=2n-1$. That is, 
\begin{equation}
 \label{eqn:homogeneity volume form}
 \mu_r^*(\Upsilon'_{X_r})=r^{2n-1}\,\Upsilon'_{X}.
\end{equation}

An alternative and common choice of a Riemannian structure $g_1$ on $X_1$ comes from its
structure of a unit circle bundle over $F_n$.
Let 
$\alpha\in \Omega^1(X_1)$ be the connection 1-form
associated to the unique compatible covariant derivative on $A$, so that $\mathrm{d}\alpha=2\,\pi_1^*(\omega_{F_n})$. 
Also, let 
\begin{equation}
 \label{eqn:horizontal and vertical tangent space}
H(X_1/F_n)=\ker (\alpha),\,\,\,
V(X_1/F_n)=\ker (\mathrm{d}\pi_1)\subseteq TX
\end{equation}
denote the horizontal and vertical tangent bundles for $\pi_1$,
respectively.
%
There is a unique Riemannian metric $g_1$ on $X_1$ such that $\pi_1$ a Riemannian submersion, and 
the $S^1$-orbits on $X_1$ have unit length.
The corresponding volume form on $X_1$ is given by
\begin{equation}
 \label{eqn:standard volume form}
\Upsilon_{X_1}=\frac{1}{(n-1)!}\,\pi_1^*\left(\omega_{F_n}^{\wedge (n-1)}\right)\,\wedge\,\frac{1}{2\pi}\,\alpha=
\frac{1}{2\pi}\,\pi_1^*(\Upsilon_{F_n})\wedge \alpha,
\end{equation}
where $\Upsilon_{F_n}=\omega_{F_n}^{\wedge (n-1)}/(n-1)!$ is the symplectic volume form on $F_n$.

We wish to compare the two Riemannian metrics $g_1$ and $g_1'$, the corresponding volume
forms, $\Upsilon_{X_1}'$ and $\Upsilon_{X_1}$, and densities, $\mathrm{d}V_X$ and $\mathrm{d}'V_X$.
 
 \begin{lem}
  \label{lem:relation volume forms and densities}
$\Upsilon_{X_1}=\frac{1}{2\pi}\,\Upsilon_{X_1}'$ and $\mathrm{d}V_{X_1}=\frac{1}{2\pi}\,\mathrm{d}'V_{X_1}$.  
 \end{lem}

\begin{proof}
 [Proof of Lemma \ref{lem:relation volume forms and densities}.]
The connection 1-form for the Hopf map $S^{2n+1}\rightarrow \mathbb{P}^n$ is 
$$
\theta=\frac{i}{2}\,\left(\mathbf{z}^t\,\mathrm{d}\overline{\mathbf{z}}^t-\overline{\mathbf{z}}^t\,\mathrm{d}\mathbf{z}\right);
$$
thus $\alpha$ is the restriction of $\theta$ to $X_1$. Let $\omega_0$ be the standard symplectic structure on
$\mathbb{C}^{n+1}$. 
Since $\theta_{\mathbf{z}}(\mathbf{w})=\omega_0(\mathbf{z},\mathbf{w})$, we have
$\ker(\theta_{\mathbf{z}})=\mathbf{z}^{\perp_{\omega_0}}$ (symplectic annihilator). 
In other words,
$$
 \ker(\theta_{\mathbf{z}})=\left(\mathrm{span}_\mathbb{R}(\mathbf{z})\oplus \mathbf{z}^{\perp_{ h_0}}\right)\cap T_{\mathbf{z}}S_1^{2n+1}=
\mathbf{z}^{\perp_{ h_0}},
$$
where $\mathbf{z}^{\perp_{ h_0}}$ is the Hermitian orthocomplement of $\mathbf{z}$ for the standard Hermitian product.

Thus, if $\mathbf{z}\in X_1$ then
$$
H_{\mathbf{z}}(X_1/F_n)=\ker(\alpha_\mathbf{z})=\mathbf{z}^{\perp_{ h_0}}\cap T_{\mathbf{z}}\mathcal{C}_n=\mathbf{z}^{\perp_{ h_0}}\cap \overline{\mathbf{z}}^{\perp_{ h_0}}.
$$
On the other hand, $V_{\mathbf{z}}(X_1/F_n)=\mathrm{span}_{\mathbb{R}}(i\mathbf{z})$. Thus $V(X_1/F_n)$ and $H(X_1/F_n)$ are orthogonal with respect
to both $g_1$ (by construction) and $g_1'$ (by the previous considerations).
Hence we may compare $g_1$ and $g_1'$ separately on $H(X_1/F_n)$ and $V(X_1/F_n)$.

On the complex vector bundle $H(X_1/F_n)$, $g_1'$ and $g_1$ are, respectively, the Euclidean scalar products associated to the 
restrcitions of the $(1,1)$-forms
$$
\omega_0=\frac{i}{2}\,\partial\overline{\partial}\|\mathbf{z}\|^2,\,\,\,\,\,\,\,\,
\omega_1=\frac{i}{2}\,\partial\overline{\partial}\ln\left(\|\mathbf{z}\|^2\right).
$$
Given that $\omega_0$ and $\omega_1$ agree on $TS_1^{2n+2}$, $g_1=g_1'$ on $H(X_1/F_n)$.

On the other hand, both $g_1$ and $g_1'$ are $S^1$-invariant, but $S^1$-orbits on $X_1$ have length $2\pi$ for $g_1'$ and $1$ for $g_1$-
Thus $g_1=g_1'/2\pi$ on $V(X_1/F_n)$.

The claim follows directly from this.
\end{proof}

\subsection{The Szeg\"{o} kernel on $X_r$}

For every $r>0$, $X_r$ is the boundary of a strictly pseudoconvex domain, and as such it carries a CR structure,
a Hardy space $H(X_r)$, and a Szeg\"{o} projector $\Pi_r:L^2(X_r)\rightarrow H(X_r)$. We aim to relate the various $\Pi_r$'s.

Let $\mathcal{O}(\mathcal{C}_n\setminus\{\mathbf{0}\})$ be the ring of holomorphic functions on the conic complex manifold
$\mathcal{C}_n\setminus\{\mathbf{0}\}$. Let $\mathcal{O}_k(\mathcal{C}_n\setminus\{\mathbf{0}\})\subset \mathcal{O}(\mathcal{C}_n\setminus\{\mathbf{0}\})$
be the subspace of holomorphic functions of degree of homogeneity $k$.

For every $r>0$ and $k=0,1,2,\ldots$ let $H_k(X_r)\subset H(X_r)$ be the finite-dimensional
$k$-th isotypical component of $H(X_r)$ with respect to the
standard $S^1$-action. Restriction induces an algebraic isomorphism $\mathcal{O}_k(\mathcal{C}_n\setminus\{\mathbf{0}\})\rightarrow H_k(X_r)$;
with a slight abuse of language, 
we shall denote by the same symbol an element of $H_k(X_r)$ and the corresponding element of $\mathcal{O}_k(\mathcal{C}_n\setminus\{\mathbf{0}\})$.

Suppose that $(s_{kj})_{j=0}^{N_k}\subseteq \mathcal{O}_k(\mathcal{C}_n\setminus\{\mathbf{0}\})$ restricts to an orthonormal basis of $H_k(X_1)$:
$$
\int_{X_1}s_{kj}(x)\,\overline{s_{kl}(x)}\,\mathrm{d}V_{X_1}(x)=\delta_{jl}.
$$
Setting $y=r\,x$, and using (\ref{eqn:homogeneity volume form}) together with Lemma \ref{lem:relation volume forms and densities}, 
we get
\begin{eqnarray}
 \label{eqn:homogeneity orthonormal basis}
\lefteqn{\int_{X_r}s_{kj}(y)\,\overline{s_{kl}(y)}\,\frac{1}{2\pi}\,\mathrm{d}'V_{X_r}(y)}\\
&=&r^{2n+2k-1}\,\int_{X_1}s_{kj}(x)\,\overline{s_{kl}(x)}\,\frac{1}{2\pi}\,\mathrm{d}'V_{X_1}(x)
=r^{2n+2k-1}\,\delta_{jl}.\nonumber
\end{eqnarray}

Therefore we have:

\begin{lem}
 \label{lem:orthonormal basis Xr}
If $(s_{kj})_{j=0}^{N_k}\subseteq \mathcal{O}_k(\mathcal{C}_n\setminus\{\mathbf{0}\})$ restricts to 
an orthonormal basis of $H_k(X_1)$ with respect to $\mathrm{d}'V_{X_1}$,
then for every $r>0$ 
$$
\left(r^{-(k+n-1/2)}\,s_{kj}\right)_{j=0}^{N_k}
$$
restricts to an orthonormal basis of $H_k(X_r)$, with respecto to $\mathrm{d}'V_{X_r}/2\pi$.
\end{lem}

Let now $\Pi_{r,k}$ be the level-$k$ Szeg\"{o} kernel on $X_r$, that is, the orthogonal projector
$$
\Pi_{r,k}:L^2(X_r,\mathrm{d}'V_{X_r}/2\pi)\rightarrow H_k(X_r).
$$
By Lemma \ref{lem:orthonormal basis Xr}, its Schwartz kernel $\Pi_{r,k}\in \mathcal{C}^\infty(X_r\times X_r)$ is given by
\begin{equation}
 \label{eqn:Szego kernel on Xr}
\Pi_{r,k}(y,y')=r^{-(2k+2n-1)}\sum_{j=0}^{N_k}s_{kj}(y)\cdot \overline{s_{kj}(y')}\,\,\,\,\,\,
(y,y'\in X_r).
\end{equation}

When pulled-back to $X_1$, this is (here $x,x'\in X_1$)
\begin{eqnarray}
 \label{eqn:Szego kernel on Xr pulled back to X}
\Pi_{r,k}(r\,x,r\,x')&=&r^{-(2k+2n-1)}\sum_{j=0}^{N_k}s_{kj}(r\,x)\cdot \overline{s_{kj}(r\,x')}\nonumber\\
&=&r^{-(2n-1)}\,\sum_{j=0}^{N_k}\widehat{s}_{kj}(x)\cdot \overline{\widehat{s}_{kj}(x')}=r^{1-2n}\,\Pi_{1,k}(x,x').
\end{eqnarray}

In particular, 
\begin{eqnarray}
 \label{eqn:Szego kernel on Xqrt2 pulled back to X}
\Pi_{\sqrt{2},k} \left(\sqrt{2}\,x,\sqrt{2}\,x'\right)=\frac{\sqrt{2}}{2^n}\,\Pi_{1,k}(x,x').
\end{eqnarray}

We shall make repeated use of the following asymptotic property of $\Pi_{1,k}$, which follows from the microlocal
description of $\Pi$ as an FIO (explicit exponential estimates are discussed in \cite{c}). 

\begin{thm}
 \label{thm:rapid decrease}
 Let 
$\mathrm{dist}_{F_n}$ be the distance function on $F_n$ associated to the K\"{a}hler metric. 
Given any $C,\,\epsilon>0$,
uniformly for $x,x'\in X$ satisfying 
$$
\mathrm{dist}_{F_n}\big(\pi(x),\,\pi(x')\big)\ge C\,k^{\epsilon-1/2},
$$
we have
$$
\Pi_{1,k}(x,x')=O\left(k^{-\infty}\right)
$$
when $k\rightarrow+\infty$. 
\end{thm}

\subsection{Heisenberg local coordinates}
\label{scn:HLC}

There are two unit circle bundles in our picture: the Hopf fibration
$\pi:S_1^{2n+1}\rightarrow \mathbb{P}^n$, and 
$\pi_1:X_1\rightarrow F_n$. Clearly, $\pi_1$ is the pull-back of $\pi$ under the inclusion
$F_n\hookrightarrow \mathbb{P}^n$. Both $S_1^{2n+1}$ and $X_1$ are boundaries of strictly pseudoconvex
domains, and carry a CR structure. 

On both $S_1^{2n+1}$ and $X_1$, we may consider privileged systems of coordinates called
\textit{Heisenberg local coordinates} (HLC).  
In these coordinates, Szeg\"{o} kernel asymptotics exhibit a \lq universal\rq\, structure \cite{sz}; 
we refer to \textit{ibidem} for a detailed discussion.

Given $\mathbf{z}_0\in X_1$, a HLC system on $X_1$ centered at $\mathbf{z}_0$ will be denoted in additive notation:
$$
(\theta,\mathbf{v})\in (-\pi,\pi)\times B_{2n-2}(\mathbf{0},\delta)\mapsto \mathbf{z}_0+(\theta,\mathbf{v})\in X_1.
$$ 
Here $\theta\in (-\pi,\pi)$ is an \lq angular\rq\, coordinate measuring displacement along the $S^1$-orbit through
$\mathbf{z}_0$ (the fiber through $\mathbf{z}_0$ of $\pi_1:X_1\rightarrow F_n$); instead $\mathbf{v}\in B_{2n-2}(\mathbf{0},\delta)
\subseteq \mathbb{R}^{2n-2}\cong \mathbb{C}^{n-1}$ descends to a local coordinate on $F_n$ centered at
$m_0=\pi(x_0)$, inducing a unitary isomorphism $T_{[\mathbf{z}_0]}F_n\cong \mathbb{C}^{n-1}$. We may thus
think of $\mathbf{v}$ as a tangent vector in $T_{[\mathbf{z}_0]}F_n$.

Here this additive notation might be misleading, 
since $X_1\subset \mathbb{C}^{n+1}$.
Therefore we shall write $\mathbf{z}_0+_{X_1}(\theta,\mathbf{v})$ for HLC's
on $X_1$ centered at $\mathbf{z}_0$. 
We shall generally abridge notation by writing $\mathbf{z}_0+_{X_1}\mathbf{v}$ for $\mathbf{z}_0+_{X_1}(0,\mathbf{v})$.

Similarly, $(\theta,\mathbf{w})\in (-\pi,\pi)\times 
B_{2n}(\mathbf{0},\delta)\mapsto \mathbf{z}_0+_{S_1^{2n+1}}(\theta,\mathbf{v})$ 
will denote a system of Heisenberg local coordinates on $S_1^{2n+1}$
centered at $\mathbf{z}_0$.
There is in fact a natural choice of HLC on $S_1^{2n+1}$ centered at any $\mathbf{z}_0\in S_1^{2n+1}$.

Namely, let $(\mathbf{a}_1,\ldots,\mathbf{a}_n)$ be an orthonormal basis of the Hermitian orthocomplement
$\mathbf{z}_0^{\perp_h}\subseteq \mathbb{C}^{n+1}$, and for $\mathbf{w}=(w_j)\in \mathbb{C}^n$ let us set 
\begin{equation}
 \label{eqn:HLC Pn}
\mathbf{z}_0+_{S_1^{2n+1}}(\theta,\mathbf{w}):=\frac{e^{i\theta}}{\sqrt{1+\|\mathbf{w}\|^2}}\,\left(\mathbf{z}_0+\sum_{j=1}^nw_j\,\mathbf{a}_j\right).
\end{equation}
Since there is a canonical unitary identification
$\mathbf{z}_0^{\perp_h}\cong T_{[\mathbf{z}_0]}\mathbb{P}^{n}$, we shall also write this as $\mathbf{z}_0+_{S_1^{2n+1}}(\theta,\mathbf{v})$
with $(\theta,\mathbf{v})\in (-\pi,\pi)\times T_{[\mathbf{z}_0]}\mathbb{P}^{n}$.



If $\mathbf{z}_0\in X_1$, HLC's on $X_1$ centered at $\mathbf{z}_0$
can be chosen so that they agree to second order with the former HLC's on $S_1^{2n+1}$. More precisely,
we may assume that for any
$\mathbf{v}\in T_{[\mathbf{z}_0]}F_n\subset T_{[\mathbf{z}_0]}\mathbb{P}^{n}$ we have
\begin{equation}
 \label{eqn:HLC X}
\mathbf{z}_0+_{X}(\theta,\mathbf{w})=\mathbf{z}_0+_{S_1^{2n+1}}\big(\theta,\mathbf{v}+R_2(\mathbf{v})\big),
\end{equation}
where $R_2$ is a function vanishing to second order at the origin.

Given $\mathbf{v},\,\mathbf{w}\in \mathbb{C}^{n+1}\cong \mathbb{R}^{2n+2}$, let us define
\begin{equation}
 \label{eqn:defn di psi2}
 \psi_2(\mathbf{v},\mathbf{w})=:-i\,\omega_0(\mathbf{v},\mathbf{w})-\frac{1}{2}\,\|\mathbf{v}-\mathbf{w}\|^2;
\end{equation}
here $\omega_0$ is the standard symplectic structure, and $\|\cdot\|$ is the standard Euclidean norm.
We shall make use of the following asymptotic expansion,
for which we refer again to \cite{sz}:

\begin{thm}
 \label{thm:sz expansion szego equiv}
Let us fix $C>0$ and $\epsilon\in (0,1/6)$. Then for any $\mathbf{z}\in X_1$, and for any choice of
HLC's on $X_1$ centered at $\mathbf{z}$, there exists polynomials $P_j$ of degree $\le 3j$ and parity $j$
on $T_{[\mathbf{z}]}F_n\times T_{[\mathbf{z}]}F_n\cong \mathbb{R}^{2n-2}\times \mathbb{R}^{2n-2}$, such that following holds.
Uniformly in $\mathbf{v}_1,\,\mathbf{v}_2\in T_{[\mathbf{z}]}F_n$ with $\|\mathbf{v}_j\|\le C\,k^{\epsilon}$ for $j=1,2$,
and $\theta_1,\,\theta_2\in (-\pi,\pi)$, one has for 
$k\rightarrow +\infty$ the following asymptotic expansion:
\begin{eqnarray*}
 \lefteqn{\Pi_{1,k}\left(\mathbf{z}+\left(\theta_1,\frac{\mathbf{v}_1}{\sqrt{k}}\right),\mathbf{z}+\left(\theta_2,\frac{\mathbf{v}_2}{\sqrt{k}}\right)\right)}\\
&\sim&\left(\frac{k}{\pi}\right)^{n-1}\,e^{ik\,(\theta_1-\theta_2)+\psi_2(\mathbf{v}_1,\mathbf{v}_2)}\,
\left[1+\sum_{j=1}^{+\infty}k^{-j/2}\,P_j(\mathbf{v}_1,\mathbf{v}_2)\right].
\end{eqnarray*}

\end{thm}

In the given range the above is an asymptotic expansion, since 
$$
\left|k^{-j/2}\,R_j(\mathbf{v}_1,\mathbf{v}_2)\right|\le C_j\,k^{-\frac{j}{2} (1-6\epsilon)}.
$$

\subsection{$\mathcal{P}_k$ and $\Pi_{\sqrt{2},k}$}

As discussed in \cite{g}, the push-forward operator
${\nu}_*:\mathcal{C}^\infty\big(X_{\sqrt{2}}\big)\rightarrow \mathcal{C}^\infty(S^n)$ restricts to an algebraic isomorphism 
\begin{equation}
 \label{eqn:push forward smooth part}
\mathcal{C}^\infty\big(X_{\sqrt{2}}\big)\cap H\big(X_{\sqrt{2}}\big)\rightarrow\mathcal{C}^\infty(S^n);
\end{equation}
for every $k$, (\ref{eqn:push forward smooth part}) restricts to a conformally unitary isomorphism 
$$
H_k(X_{\sqrt{2}})\longrightarrow V_k,
$$ 
with a
scalar conformal factor $C_{k,n}>0$. Thus we have
 \begin{equation}
 \label{eqn:conformally unitary push forward}
\|{\nu}_*(s)\|_{L^2(S^n)}=C_{k,n}\,\|s\|_{H(X_{\sqrt{2}})}\,\,\,\,\,(s\in H_k(X_{\sqrt{2}})).
\end{equation}


Therefore, if $(\sigma_{kj})_{j=0}^{N_k}$ is an orthonormal basis
of $H_k(X_{\sqrt{2}})$, then 
$$\left(C_{k,n}^{-1}\cdot {\nu}_*(\sigma_{kj})\right)_{j=0}^{N_k}$$ is
an orthonormal basis of $V_k$. It follows that $\mathcal{P}_{k,n}$ 
in (\ref{eqn:projector}) is given by 

\begin{eqnarray}
 \label{eqn:orthogonal projector as push-forward}
 \mathcal{P}_{k,n}=\frac{1}{C_{k,n}^2}\,(\nu\times \nu)_*\left(\Pi_{{\sqrt{2}},k}\right),
\end{eqnarray}
where
$\nu\times \nu:X_{\sqrt{2}}\times X_{\sqrt{2}}\rightarrow S^n\times S^n$ is the 
product projection.

More explicitly, for
$\mathbf{q}\in S^n$ let $S(\mathbf{q}^\perp)\cong S^{n-1}$ be the unit sphere centered at the origin in the
orthocomplement $\mathbf{q}^\perp$, and let
$\mathrm{d}V_{S(\mathbf{q}^\perp)}$ be the Riemannian density on $S(\mathbf{q}^\perp)$; then 
\begin{eqnarray}
 \label{eqn:push forward explicit}
 \lefteqn{\mathcal{P}_{k,n}(\mathbf{q}_0,\mathbf{q}_1)}\\
 &=&\frac{1}{C_{k,n}^2}\,\int_{S(\mathbf{q}_0^\perp)}\,\int _{S({\mathbf{q}_1}^\perp)}\,
 \Pi_{\sqrt{2},k}(\mathbf{q}_0+i\,\mathbf{p},\mathbf{q}_1+i\,\mathbf{p}')\,\mathrm{d}V_{S(\mathbf{q}_0^\perp)}(\mathbf{p})\,
 \mathrm{d}V_{S({\mathbf{q}_1}^\perp)}(\mathbf{p}').
 \nonumber
\end{eqnarray}

\subsection{$\Pi_{r,k}$ and conjugation}

Conjugation $\sigma:\mathbf{z}\mapsto \overline{\mathbf{z}}$ in $\mathbb{C}^{n+1}$ leaves invariant
the affine cone $\mathcal{C}_n$ and every
$X_r$. Furthermore, it yields a Riemannian isometry of $X_r$ into itself.
For $f\in \mathcal{O}\left(\mathcal{C}_n\setminus\{\mathbf{0}\}\right)$,
let us set 
$$
f^\sigma(\mathbf{z})=:\overline{f\left(\overline{\mathbf{z}}\right)}.
$$
 If $f\in \mathcal{O}_k\left(\mathcal{C}_n\setminus\{\mathbf{0}\}\right)$,
then $f^\sigma\in \mathcal{O}_k\left(\mathcal{C}_n\setminus\{\mathbf{0}\}\right)$.

Hence, if $(s_{kj})_j\subseteq \mathcal{O}_k\left(\mathcal{C}_n\setminus\{\mathbf{0}\}\right)$ 
restricts to an orthonormal basis of $H_k\big(X_r\big)$, then
so does $(s_{kj}^\sigma)_j$. Thus for any $\mathbf{z}_0,\,\mathbf{z}_1\in X_r$ we have
 \begin{eqnarray}
  \label{eqn:self conjugate szego kernel}
  \Pi_{r,k}\left(\overline{\mathbf{z}}_0,\overline{\mathbf{z}}_1\right)&=&
  \sum_j s_{kj}\left(\overline{\mathbf{z}}_0\right)\cdot \overline{s_{kj}\left(\overline{\mathbf{z}}_1\right)}\nonumber\\
  &=&\sum_j \overline{s_{kj}^\sigma\left(\mathbf{z}_0\right)}\cdot s_{kj}^\sigma\left(\mathbf{z}_1\right)=
  \Pi_{rk}\left(\mathbf{z}_1,\mathbf{z}_0\right)=\overline{\Pi_{rk}\left(\mathbf{z}_0,\mathbf{z}_1\right)}.
 \end{eqnarray}

 \section{Proof of Theorem \ref{thm:main expansion}}

\begin{proof}
 [Proof of Theorem \ref{thm:main expansion}.]
Given $\mathbf{q}_0,\,\mathbf{q}_1\in S^{n-1}$ with 
$\mathbf{q}_1\neq \pm \mathbf{q}_0$, let 
$\gamma_+$ be the unique unit speed geodesic on $S^n$ such that 
$\gamma_+(0)=\mathbf{q}_0$ and $\gamma_+(\vartheta)=\mathbf{q}'$ for
some $\vartheta\in (0,\pi)$. 
Then 
$$
\mathbf{p}_0=:\dot{\gamma}_+(0)\in S^{n-1}\big(\mathbf{q}^\perp_0\big),
\,\,\,\,\mathbf{p}_1=:\dot{\gamma}_+(\vartheta)\in S^{n-1}\big(\mathbf{q}^\perp_1\big).
$$

The reverse geodesic $\gamma_-(\vartheta)=:\gamma(-\vartheta)$ satisfies
$\gamma_-(0)=\mathbf{q}$, $\dot{\gamma}_-(0)=-\mathbf{p}_0$ and
$\gamma_-(\vartheta')=\mathbf{q}'$ for a unique $\vartheta'=-\vartheta\in (-\pi,0)$.

Although they project down to the same locus in $S^n$,
$\gamma_+$ and $\gamma_-$ correspond to distinct fibers of the circle bundle
projection $\pi:X(\sqrt{2})\rightarrow F_n$.
Let us express the (co)tangent lift $\widetilde{\gamma}_{\pm}$
of the geodesics $\gamma_{\pm}$ in complex coordinates, and  set $\mathbf{p}_1=\dot{\gamma}_+(\vartheta)$. .
Then

\begin{equation}
 \label{eqn:relation geodesic flow +-}
\widetilde{\gamma}_{\pm}(\theta)=\gamma_{\pm}(\theta)+i\,\dot{\gamma}_{\pm}(\theta)
=e^{-i\theta}\,(\mathbf{q}_0\pm i\,\mathbf{p}_0)
=\mathbf{q}_1\pm i \mathbf{p}_1,
\end{equation}
In view of (\ref{eqn:geodesic flow and circle bundle}), we have:
\begin{equation}
 \label{eqn:same fiber of pi pm}
\mathbf{q}_0\pm i\,\mathbf{p}_0,\,\mathbf{q}_1\pm i\,\mathbf{p}_1\in \pi_{\sqrt{2}}^{-1}\big([\mathbf{q}_0\pm i\,\mathbf{p}_0]\big).
\end{equation}
On the other hand, 
$[\mathbf{q}_0+i\,\mathbf{p}_0]\neq [\mathbf{q}_0-i\,\mathbf{p}_0]\in F_n$,
since $\mathbf{q}_0+i\,\mathbf{p}_0$ and $\mathbf{q}_0-i\,\mathbf{p}_0$ are linearly independent in
$\mathbb{C}^{n+1}$.

Thus we have:

\begin{lem}
 \label{lem:unique fibers}
Suppose $\mathbf{q}_0,\,\mathbf{q}_1\in S^n$ and $\mathbf{q}_1\neq \pm \mathbf{q}_0$.
Then the only points $[\mathbf{z}]\in F_n$ such that 
$$
\nu^{-1}(\mathbf{q}_0)\cap \pi_{\sqrt{2}}^{-1}([\mathbf{z}])\neq \emptyset\,\,\,\,
\mathrm{and}\,\,\,\nu^{-1}(\mathbf{q}_1)\cap \pi_{\sqrt{2}}^{-1}([\mathbf{z}])\neq \emptyset
$$
are 
$$
[\mathbf{z}_+]=[\mathbf{q}_0+i\,\mathbf{p}_0],\,\,\,\,
[\mathbf{z}_-]=[\mathbf{q}_0-i\,\mathbf{p}_0].
$$
\end{lem}

By Theorem \ref{thm:rapid decrease}, for fixed 
$\mathbf{p}$ and $\mathbf{p}'$ and
$k\rightarrow +\infty$ we have
$$
\Pi_{\sqrt{2},k}(\mathbf{q}_0+i\,\mathbf{p},\mathbf{q}_1+i\,\mathbf{p}')=O\left(k^{-\infty}\right),
$$
unless $\mathbf{p}=\pm\mathbf{p}_0$ and $\mathbf{p}'=\pm\mathbf{p}_1$.
Therefore, for a fixed $\vartheta\in (0,\pi)$ integration in (\ref{eqn:push forward explicit}) may be localized in a 
small neighborhood
of $(\pm\mathbf{p}_0,\pm\mathbf{p}_1)$, perhaps at the cost of disregarding a negligible contribution to the asymptotics.

Since however we are allowing $\vartheta$ to approach $0$ or $\pi$ at a controlled rate, we need to give a more
precise quantitative estimate of how small the previous neighborhood may be chosen when $k\rightarrow+\infty$.

To this end, let us introduce some further notation.
Given linearly independent $\mathbf{a},\, \mathbf{b}\in S^n$,
let us set
$$
R(\mathbf{a},\mathbf{b})=:
\mathrm{span}_\mathbb{R}(\mathbf{a},\mathbf{b})\subseteq \mathbb{R}^{n+1},
$$
and
$$
R(\mathbf{a},\mathbf{b})_{\mathbb{C}}=:R(\mathbf{a},\mathbf{b})\otimes \mathbb{C}=
\mathrm{span}_\mathbb{C}(\mathbf{a},\mathbf{b})\subseteq \mathbb{C}^{n+1},
$$
Furthermore, for $\|\mathbf{v}\|\le 1$ we shall set
$$
S_{\pm}(\mathbf{v})=:-1\pm\sqrt{1-\|\mathbf{v}\|^2}.
$$

A straightforward computation yields the following:

\begin{lem}
 \label{lem:Spm}
Assume that $\mathbf{q}_0+i\,\mathbf{p}_0\in X_{\sqrt{2}}$ and
$\mathbf{q}_1+i\,\mathbf{p}_1=e^{-i\vartheta}\,(\mathbf{q}_0+i\,\mathbf{p}_0)$
with $\vartheta\in (0,\pi)$.
Then any
$\mathbf{p}\in S^{n-1}(\mathbf{q}_0^\perp)$ with 
$\mathbf{p}_0^t\,\mathbf{p}\ge 0$, respectively, $\mathbf{p}_0^t\,\mathbf{p}\le 0$,
may be written uniquely
in the form
\begin{equation}
 \label{eqn:p in termini di v +}
 \mathbf{p}=\big(1+S_+(\mathbf{v})\big)\,\mathbf{p}_0+\mathbf{v},
\end{equation}
respectively
\begin{equation}
 \label{eqn:p in termini di v -}
\mathbf{p}=\big(1+S_-(\mathbf{v})\big)\,\mathbf{p}_0+\mathbf{v},
\end{equation}
where $\mathbf{v}\in \mathbf{q}_0^\perp\cap \mathbf{p}_0^\perp=R(\mathbf{q}_0,\mathbf{q}_1)^\perp$ (the Euclidean
orthocomplement) has norm $\le 1$, and 
$$
S_{\pm}(\mathbf{v})=:-1\pm\sqrt{1-\|\mathbf{v}\|^2}.
$$ 
\end{lem}

\begin{prop}
\label{prop:shrinking delta}
Let us fix $C>0$, $\delta\in (0,1/6)$ and $\epsilon>\delta$.
Then there exist constants $D,\,\epsilon_1>0$ such that the following holds.
Suppose that
\begin{enumerate}
 \item $C\,k^{-\delta}<\vartheta<\pi-C\,k^{-\delta}$;
 \item $\mathbf{q}_j+i\,\mathbf{p}_j\in X_{\sqrt{2}}$ for $j=0,1$;
 \item $\mathbf{q}_1+i\,\mathbf{p}_1=e^{-i\vartheta}\,(\mathbf{q}_0+i\,\mathbf{p}_0)$;
\item $\mathbf{v}_j\in \mathbf{q}_0^\perp\cap \mathbf{q}_1^\perp$ for $j=0,1$;
\item $1\ge \max\big\{\|\mathbf{v}_0\|,\,\|\mathbf{v}_1\|\big\}\ge C\,k^{\epsilon-1/2}$;
\item $\mathbf{p}'_j=\big(1+S_j(\mathbf{v}_j)\big)\,\mathbf{p}_j+\mathbf{v}_j\in S^{n-1}(\mathbf{q}_j^\perp)$ for $j=0,1$,
where $S_j$ can be either one of $S_{\pm}$ (Lemma \ref{lem:Spm}).
\end{enumerate}

Then 
 $$
 \mathrm{dist}_{F_n}\big([\mathbf{q}_0+i\,\mathbf{p}_0'],\,[\mathbf{q}_1+i\,\mathbf{p}'_1]\big)\ge D\,k^{\epsilon_1-1/2}
 $$
for every $k\gg 0$.
\end{prop}

In view of Theorem \ref{thm:rapid decrease}, Proposition \ref{prop:shrinking delta} implies:

\begin{cor}
 \label{cor:rapid decrease}
 Uniformly in the range of Proposition \ref{prop:shrinking delta}, we have
 $$
 \Pi_{\sqrt{2},k}(\mathbf{q}_0+i\,\mathbf{p}_0',\mathbf{q}_1+i\,\mathbf{p}'_1)=O\left(k^{-\infty}\right).
 $$
\end{cor}

\begin{proof}
[Proof of Proposition \ref{prop:shrinking delta}]
Let us set for $\gamma\in [-\pi,\pi]$:
\begin{equation}
 \label{eqn:basic difference defn}
\Phi(\gamma,\mathbf{p}_0',\mathbf{p}'_1)=: e^{-i\gamma}\,(\mathbf{q}_0+i\mathbf{p}_0')-(\mathbf{q}_1+i\,\mathbf{p}'_1).
\end{equation}
Let 
$\mathrm{dist}_{F_n}$ be the restriction to $F_n$ of the distance function on $\mathbb{P}^n$. 
Then
\begin{eqnarray}
 \label{eqn:distance as min}
 \lefteqn{\mathrm{dist}_{F_n}\big([\mathbf{q}_0+i\,\mathbf{p}_0'],\,[\mathbf{q}_1+i\,\mathbf{p}'_1]\big)}\\
 &=&\frac{1}{\sqrt{2}}\,
 \min\big\{\|\Phi(\gamma,\mathbf{p}_0',\mathbf{p}'_1)\|\,:\,\gamma\in [0,2\pi]\big\}.\nonumber
\end{eqnarray}
The factor in front is needed because while the Hopf map  
$S^{2n+1}_1\rightarrow \mathbb{P}^n$ is a Riemannian submersion,
the projection $S^{2n+1}_{\sqrt{2}}\rightarrow \mathbb{P}^n$ 
is so only in a conformal sense.

We are reduced to proving that in the given range there exist constants 
$D,\,\epsilon_1>0$ such that for every $k\gg 0$ and $\gamma\in [0,2\pi]$
\begin{equation}
 \label{eqn:basic estimate on Phi}
 \|\Phi(\gamma,\mathbf{p}_0',\mathbf{p}'_1)\|\ge D\,k^{\epsilon_1-1/2}.
 \end{equation}

We have
\begin{eqnarray}
 \label{eqn:basic difference}
\lefteqn{\Phi(\gamma,\mathbf{p}_0',\mathbf{p}'_1)}\nonumber\\
&=&e^{-i\gamma}\,\Big(\mathbf{q}_0+i\mathbf{p}_0+i\,S_0(\mathbf{v}_0)\,\mathbf{p}_0+i\mathbf{v}_0\Big)-
\left(e^{-i\vartheta}\,(\mathbf{q}_0+i\mathbf{p}_0)+i\,S_1(\mathbf{v}_1)\,\mathbf{p}_1+i\mathbf{v}_1\right)\nonumber\\&=&(A\,\mathbf{q}_0+B\,\mathbf{p}_0)+i\,\left[e^{-i\gamma}\,\mathbf{v}_0-\mathbf{v}_1\right],
\end{eqnarray}
where
\begin{eqnarray}
 \label{eqn:defn di A}
A&=:&\left(e^{-i\gamma}-e^{-i\vartheta}\right)+i\,S_1(\mathbf{v}_1)\,\sin(\vartheta)\nonumber\\
&=&\cos(\gamma)-\cos(\vartheta)+i\,\left[-\sin(\gamma)+\sin(\vartheta)\big(1+S_1(\mathbf{v}_1)\big)\right],
\end{eqnarray}
\begin{eqnarray}
 \label{eqn:defn di B}
B&=:&i\,\left(e^{-i\gamma}-e^{-i\vartheta}\right)
+i\,\left(e^{-i\gamma}\,S_0(\mathbf{v}_0)-S_1(\mathbf{v}_1)\,\cos(\vartheta)\right)\nonumber\\
&=&\sin(\gamma)\,\big(1+S_0(\mathbf{v}_0)\big)-\sin(\vartheta)\nonumber\\
&&+i\,\big(\cos(\gamma)\,S_0(\mathbf{v}_0)-S_1(\mathbf{v}_1)\,\cos(\vartheta)+
\cos(\gamma)-\cos(\vartheta)\big).
\end{eqnarray}

Regarding the two summands on the last line of (\ref{eqn:basic difference}), we have
$$
A\,\mathbf{q}_0+B\,\mathbf{p}_0\in R(\mathbf{q}_0,\mathbf{q}_1)_\mathbb{C},\,\,\,\,\,\,\,
i\,\left[e^{-i\gamma}\,\mathbf{v}_0-\mathbf{v}_1\right]\in R(\mathbf{q}_0,\mathbf{q}_1)_\mathbb{C}^{\perp_h},
$$
where $\perp_h$ denotes the Hermitian orthocomplement.
Hence
\begin{eqnarray}
 \label{eqn:lower bound on Phi 1}
 \|\Phi(\gamma,\mathbf{p}_0',\mathbf{p}'_1)\|^2&\ge& \left\|e^{-i\gamma}\,\mathbf{v}_0-\mathbf{v}_1\right\|^2\nonumber\\
&\ge&\big(1-|\cos(\gamma)|\big)\,\left[\|\mathbf{v}_0\|^2+\|\mathbf{v}_1\|^2\right].
\end{eqnarray}

Since $1-|\cos(\gamma)|$ vanishes exactly to second order at $\gamma=0,\,\pi,\,2\pi$, there exists 
$D>0$ such that for $\gamma\in [0,2\,\pi]$ we have
$$
1-|\cos(\gamma)|\ge D^2\,\min\left\{\gamma^2,\,(\gamma-\pi)^2,\,(\gamma-2\pi)^2\right\}.
$$
Given this and (\ref{eqn:lower bound on Phi 1}),
we conclude that, under the present hypothesis,
\begin{eqnarray}
 \label{eqn:lower bound on Phi 2}
 \|\Phi(\gamma,\mathbf{p}_0',\mathbf{p}_1')\|&\ge& D\,\min\left\{\gamma,\,|\gamma-\pi|,2\pi-\gamma\right\}\,
 \max\big\{\|\mathbf{v}\|,\|\mathbf{v}'\|\big\}\nonumber\\
 &\ge& 
 C\,D\,\min\left\{\gamma,\,|\gamma-\pi|,\,2\pi-\gamma\right\}\,k^{\epsilon-1/2}.
\end{eqnarray}
Let us now pick $\delta'$ with $\epsilon>\delta'>\delta$, and assume 
\begin{equation}
 \label{eqn:first bound on gamma}
 \min\left\{\gamma,\,|\gamma-\pi|,\,2\pi-\gamma\right\}\ge k^{-\delta'}.
\end{equation}

Then
\begin{eqnarray}
 \label{eqn:lower bound on Phi 3}
 \|\Phi(\gamma,\mathbf{p},\mathbf{p}')\|\ge
 C\,D\,k^{(\epsilon-\delta')-1/2}.
\end{eqnarray}
This establishes (\ref{eqn:basic estimate on Phi}) with $\epsilon_1=\epsilon-\delta'$,
in the case where (\ref{eqn:first bound on gamma})
holds. Thus we are reduced to assuming 
\begin{equation}
 \label{eqn:second bound on gamma}
 \min\left\{\gamma,\,|\gamma-\pi|,\,2\pi-\gamma\right\}\le k^{-\delta'}.
\end{equation}
Then we also have 
$|\sin(\gamma)|\le k^{-\delta'}$. 
Let us then look at the first summand on the last line of (\ref{eqn:basic difference}). We
have an Hermitian orthogonal direct sum 
$$
R(\mathbf{q}_0,\mathbf{q}_1)_\mathbb{C}=R(\mathbf{q}_0,\mathbf{p}_0)_\mathbb{C}=\mathrm{span}_{\mathbb{C}}(\mathbf{q}_0)
\oplus \mathrm{span}_{\mathbb{C}}(\mathbf{p}_0).
$$
On the other hand, since $\sin(\vartheta)$ vanishes exactly to first order 
at $\vartheta=0$ and $\vartheta=\pi$, there exists $E>0$ such that for
$\vartheta\in (0,\pi)$ under the assumptions of the Lemma we have 
$$
\sin(\vartheta)\ge E\,\min\big\{\vartheta,\,\pi-\vartheta\}\ge E\,C\,k^{-\delta}
$$
Hence, in view of (\ref{eqn:defn di B}), we have for some $D_1>0$ and $k\gg 0$
\begin{eqnarray}
 \label{eqn:bound AB}
 \|\Phi(\gamma,\mathbf{p},\mathbf{p}')\|&\ge& |A\,\mathbf{q}_0+B\,\mathbf{p}_0|\ge |B|\ge |\Re (B)|\nonumber\\
 &=&\left|\sin(\gamma)\,\big(1+S_0(\mathbf{v}_0)\big)-\sin(\vartheta)\right|
 \ge\left|\sin(\vartheta)\right|-k^{-\delta'} \nonumber\\
 &\ge&E\,C\,k^{-\delta}-k^{-\delta'}\ge \frac{1}{2}\, E\,C\,k^{-\delta}\ge \frac{1}{2}\, E\,C\,k^{-1/6},
\end{eqnarray}
since $\delta'>\delta$ and $\delta<1/6$. 
This establishes (\ref{eqn:basic estimate on Phi}) with $\epsilon_1=1/3$
when (\ref{eqn:second bound on gamma}) holds.

The proof of Proposition \ref{prop:shrinking delta} is complete.
 \end{proof}

Equations (\ref{eqn:p in termini di v +}) and
(\ref{eqn:p in termini di v -}) parametrize neighborhoods of $\mathbf{p}_0$ and $-\mathbf{p}_0$, respectively.
Therefore, Proposition \ref{prop:shrinking delta} implies that in (\ref{eqn:push forward explicit})
only a negligible contribution to the asymptotics is
lost, if integration in $\mathbf{p}$ and $\mathbf{p}'$ is restricted to shrinking neighborhoods of $\pm \mathbf{p}_0$
and $\pm \mathbf{p}_1$, of radii $O\left(k^{\epsilon-1/2}\right)$. 

This may be rephrased as follows.
Let $\varrho\in \mathcal{C}^\infty_0\left(\mathbb{R}^{n+1}\right)$ 
be even, supported in a small neighborhood of the origin, and identically equal to one in a smaller neighborhood
of the origin. Then the asymptotics of (\ref{eqn:push forward explicit}) are unchanged,
if the integrand is multiplied by 
\begin{eqnarray}
 \label{eqn:rescaled cut-off}
 \lefteqn{\left[\varrho\left(k^{1/2-\epsilon}\,(\mathbf{p}-\mathbf{p}_0)\right)
 +\varrho\left(k^{1/2-\epsilon}\,(\mathbf{p}+\mathbf{p}_0)\right)\right]}\\
&&\cdot \left[\varrho\left(k^{1/2-\epsilon}\,(\mathbf{p}'-\mathbf{p}_1)\right)
+\varrho\left(k^{1/2-\epsilon}\,(\mathbf{p}'+\mathbf{p}_1)\right)\right].\nonumber
\end{eqnarray}

In this way the integrand splits into four summands. In fact, only two of these
are non-negligible for $k\rightarrow+\infty$. Namely, consider the summand containing the factor
\begin{equation}
 \label{eqn:negligible factor 1}
 \varrho\left(k^{1/2-\epsilon}\,(\mathbf{p}-\mathbf{p}_0)\right)\,\varrho\left(k^{1/2-\epsilon}\,(\mathbf{p}'+\mathbf{p}_1)\right).
\end{equation}

On its support, $\mathbf{p}$ lies in a shrinking neighborhood of $\mathbf{p}_0$, and $\mathbf{p}'$
in a shrinking neighborhood of $-\mathbf{p}_1$. Therefore, on the same support
$\mathbf{q}_0+i\,\mathbf{p}$ lies in a shrinking neighborhood of $\mathbf{q}_0+i\,\mathbf{p}_0$, and
$\mathbf{q}_1-i\,\mathbf{p}'$ lies in a shrinking neighborhood of $\mathbf{q}_1-i\,\mathbf{p}_1$.
Since
\begin{eqnarray*}
\frac{1}{\sqrt{2}}\,(\mathbf{q}_0+i\,\mathbf{p}_0)\wedge \frac{1}{\sqrt{2}}\,(\mathbf{q}_1-i\,\mathbf{p}_1)&=&
\frac{1}{2}\,(\mathbf{q}_0+i\,\mathbf{p}_0)\wedge e^{i\vartheta}\,(\mathbf{q}_0-i\,\mathbf{p}_0)\\
&=&-i\,\mathbf{q}_0\wedge\mathbf{p}_0
\end{eqnarray*}
has unit norm, on the support of (\ref{eqn:negligible factor 1})
$[\mathbf{q}_0+i\,\mathbf{p}]$ and $[\mathbf{q}_1+i\,\mathbf{p}']$ remain at a distance $\ge 2/3$, say, in projective
space.  
This implies that as $k\rightarrow +\infty$
$$
\Pi_{\sqrt{2}, k}(\mathbf{q}_0+i\,\mathbf{p},\mathbf{q}_1+i\,\mathbf{p}')=O\left(k^{-\infty}\right)
$$
uniformly in $(\mathbf{p},\,\mathbf{p}')$ in the support of (\ref{eqn:negligible factor 1}).
A similar argument applies to the summand containing the factor
\begin{equation}
 \label{eqn:negligible factor 2}
 \varrho\left(k^{1/2-\epsilon}\,(\mathbf{p}+\mathbf{p}_0)\right)\,\varrho\left(k^{1/2-\epsilon}\,(\mathbf{p}'-\mathbf{p}_1)\right).
\end{equation}

Thus we may rewrite (\ref{eqn:push forward explicit}) as follows:
\begin{equation}
 \label{eqn:push forward split}
 \mathcal{P}_{k,n}(\mathbf{q}_0,\mathbf{q}_1)\sim \mathcal{P}_{k,n}(\mathbf{q}_0,\mathbf{q}_1)_{+}+
 \mathcal{P}_{k,n}(\mathbf{q}_0,\mathbf{q}_1)_{-},
\end{equation}
where
\begin{eqnarray}
 \label{eqn:push forward explicit k+}
 \mathcal{P}_{k,n}(\mathbf{q}_0,\mathbf{q}_1)_{\pm}
 &=:&\frac{1}{C_{k,n}^2}\,\int_{S(\mathbf{q}_0^\perp)}\,\int _{S({\mathbf{q}_1}^\perp)}\,
 \varrho\left(k^{1/2-\epsilon}\,(\mathbf{p}\mp\mathbf{p}_0)\right)\,\varrho\left(k^{1/2-\epsilon}\,(\mathbf{p}'\mp\mathbf{p}_1)\right)
 \nonumber\\
 &&\cdot 
 \Pi_{\sqrt{2},k}(\mathbf{q}_0+i\,\mathbf{p},\mathbf{q}_1+i\,\mathbf{p}')\,\mathrm{d}V_{S(\mathbf{q}_0^\perp)}(\mathbf{p})\,
 \mathrm{d}V_{S({\mathbf{q}_1}^\perp)}(\mathbf{p}').
\end{eqnarray}

As a further reduction, we need only deal with one of $\mathcal{P}_{k,n}(\mathbf{q}_0,\mathbf{q}_1)_{\pm}$.

\begin{lem}
 \label{lem:pm are conjugate}
 $\mathcal{P}_{k,n}(\mathbf{q}_0,\mathbf{q}_1)_{\pm}=\overline{\mathcal{P}_k(\mathbf{q}_0,\mathbf{q}_1)_{\mp}}$.
\end{lem}

\begin{proof}
[Proof of Lemma \ref{lem:pm are conjugate}].
Let us apply the change of integration variable $\mathbf{p}\mapsto -\mathbf{p}$ and $\mathbf{p}'\mapsto -\mathbf{p}'$,
and apply (\ref{eqn:self conjugate szego kernel}). Since $\varrho$ is even, we get
\begin{eqnarray*}
 \mathcal{P}_{k,n}(\mathbf{q}_0,\mathbf{q}_1)_{-}
 &=&\frac{1}{C_{k,n}^2}\,\int_{S(\mathbf{q}_0^\perp)}\,\int _{S({\mathbf{q}_1}^\perp)}\,
 \varrho\left(k^{1/2-\epsilon}\,(\mathbf{p}-\mathbf{p}_0)\right)\,\varrho\left(k^{1/2-\epsilon}\,(\mathbf{p}'-\mathbf{p}_1)\right)
 \nonumber\\
 &&\cdot 
 \overline{\Pi_{\sqrt{2},k}(\mathbf{q}_0+i\,\mathbf{p},\mathbf{q}_1+i\,\mathbf{p}')}\,\mathrm{d}V_{S(\mathbf{q}_0^\perp)}(\mathbf{p})\,
 \mathrm{d}V_{S({\mathbf{q}_1}^\perp)}(\mathbf{p}')\\
 &=&\overline{\mathcal{P}_k(\mathbf{q}_0,\mathbf{q}_1)_{+}}.
\end{eqnarray*}

\end{proof}

Lemma \ref{lem:pm are conjugate} and (\ref{eqn:push forward split}) imply
\begin{equation}
 \label{eqn:push forward real part}
 \mathcal{P}_{k,n}(\mathbf{q}_0,\mathbf{q}_1)\sim 2\,\Re\big(\mathcal{P}_k(\mathbf{q}_0,\mathbf{q}_1)_{+}\big).
\end{equation}

In the definition of $\mathcal{P}_k(\mathbf{q}_0,\mathbf{q}_1)_{+}$, integration is over a shrinking neighborhood
of $(\mathbf{p}_0,\mathbf{p}_1)\in S(\mathbf{q}_0^\perp)\times S({\mathbf{q}_1}^\perp)$. We can thus make use of 
the parametrization (\ref{eqn:p in termini di v +}), and write in (\ref{eqn:push forward explicit k+}):
$$
\mathbf{p}=
\mathbf{p}_0+A(\mathbf{v}_0),\,\,\,\,\,\,
\mathbf{p}'
=\mathbf{p}_1+A(\mathbf{v}_1),
$$
where we have set
$$
A(\mathbf{v}_j)=:\mathbf{v}_j+S_+(\mathbf{v}_j)\,\mathbf{p}_j.
$$
It is also harmless to replace $\mathbf{p}-\mathbf{p}_j$ by $\mathbf{v}_j$ in the rescaled cut-offs
in (\ref{eqn:push forward explicit k+}). Let us also set $\mathbf{z}_j=\mathbf{q}_j+i\,\mathbf{p}_j$, and recall that 
$\mathbf{z}_1=e^{-i\vartheta}\,\mathbf{z}_0$. We then obtain
\begin{eqnarray}
 \label{eqn:push forward explicit k+1}
 \mathcal{P}_{k,n}(\mathbf{q}_0,\mathbf{q}_1)_{+}
 &=&\frac{1}{C_{k,n}^2}\,\int_{\mathbf{q}_0^\perp\cap \mathbf{q}_1^\perp}\,
 \int _{\mathbf{q}_0^\perp\cap \mathbf{q}_1^\perp}\,
 \varrho\left(k^{1/2-\epsilon}\,\mathbf{v}_0\right)\,\varrho\left(k^{1/2-\epsilon}\,\mathbf{v}_1\right)\\
 &&\cdot 
 \Pi_{\sqrt{2},k}\Big(\mathbf{z}_0+i\,A(\mathbf{v}_0),\mathbf{z}_1+i\,A(\mathbf{v}_1)\Big)\cdot \mathcal{V}(\mathbf{v}_0,\mathbf{v}_1)\,\mathrm{d}\mathbf{v}_0\,
 \mathrm{d}\mathbf{v}_1,\nonumber
\end{eqnarray}
where $\mathcal{V}(\mathbf{0},\mathbf{0})=1$.

Let us pass to rescaled integration variables $\mathbf{v}_j\mapsto \mathbf{v}_j/\sqrt{k}$ in (\ref{eqn:push forward explicit k+1}). 
Then 
\begin{eqnarray}
 \label{eqn:push forward explicit k+rescaled}
 \mathcal{P}_{k,n}(\mathbf{q}_0,\mathbf{q}_1)_{+}
 &=&\frac{k^{1-n}}{C_{k,n}^2}\,\int_{\mathbf{q}_0^\perp\cap \mathbf{q}_1^\perp}\,
 \int _{\mathbf{q}_0^\perp\cap \mathbf{q}_1^\perp}\,
 \varrho\left(k^{-\epsilon}\,\mathbf{v}_0\right)\,\varrho\left(k^{-\epsilon}\,\mathbf{v}_1\right)\nonumber
\\
 &&\cdot 
 \Pi_{\sqrt{2},k}\left(\mathbf{z}_0+\frac{i}{\sqrt{k}}\,A_{0k}(\mathbf{v}_0),
 \mathbf{z}_1+\frac{i}{\sqrt{k}}\,A_{1k}(\mathbf{v}_1)\right)\nonumber\\
&&\cdot \mathcal{V}\left(\dfrac{\mathbf{v}_0}{\sqrt{k}},\frac{\mathbf{v}_1}{\sqrt{k}}\right)\,\mathrm{d}\mathbf{v}_0\,
 \mathrm{d}\mathbf{v}_1, 
\end{eqnarray}
with 
\begin{eqnarray}
\label{eqn:defn of Ajk}
 A_{jk}(\mathbf{v}):=\mathbf{v}+\sqrt{k}\cdot S_+\left(\frac{\mathbf{v}}{\sqrt{k}}\right)\,\mathbf{p}_j.
\end{eqnarray}

Let us consider the Szeg\"{o} term in the integrand. In view of (\ref{eqn:Szego kernel on Xqrt2 pulled back to X}),
this is
\begin{eqnarray}
 \label{eqn:szego term integrand}
\lefteqn{\Pi_{\sqrt{2},k}\left(\mathbf{z}_0+\frac{i}{\sqrt{k}}\,\,A_{0k}(\mathbf{v}_0),\mathbf{z}_1
+\frac{i}{\sqrt{k}}\,\,A_{1k}(\mathbf{v}_1)\right)}\\
&=&\frac{\sqrt{2}}{2^n}\,e^{ik\vartheta}\,\Pi_{1,k}\left(\frac{\mathbf{z}_0}{\sqrt{2}}+\frac{1}{\sqrt{k}}\,
\frac{i\,\,A_{0k}(\mathbf{v}_0)}{\sqrt{2}},
\frac{\mathbf{z}_0}{\sqrt{2}}+\frac{1}{\sqrt{k}}\,\frac{i\,e^{i\vartheta}\,\,A_{1k}(\mathbf{v}_1)}{\sqrt{2}}\right)\nonumber
\end{eqnarray}

Now the sums in the previous expression are just algebraic sums in $\mathbb{C}^{n+1}$; 
in order to apply the scaling asymptotics of Theorem \ref{thm:sz expansion szego equiv}, we need to first express 
the argument of (\ref{eqn:szego term integrand}) in terms of local Heisenberg coordinates on $X_1$ centered at $\mathbf{z}_0/\sqrt{2}$.

\begin{lem}
 \label{lem:algebraic sum and HLC}
Suppose 
$\mathbf{z}=\mathbf{q}+i\,\mathbf{p}\in X_{1}$ 
and choose a system of HLC's on $X_{1}$ centered at $\mathbf{z}$.
Then for $\delta\mathbf{p}\sim \mathbf{0}\in \mathbb{R}^{n+1}$ and $e^{i\vartheta}\in S^1$
such that
$\mathbf{z}+i\,e^{i\vartheta}\,\delta\mathbf{p}\in X_{1}$ we have
$$
\mathbf{z}+i\,e^{i\vartheta}\,\delta\mathbf{p}=\mathbf{z}+_{X_{1}}\left(0,i\,e^{i\vartheta}\,\delta\mathbf{p}+
R_2(\theta;\delta\mathbf{p})\right),
$$
for a suitable smooth function $R_2(\theta;\cdot)$ vanishing to second order at the origin (in $\mathbf{v}$).
\end{lem}

\begin{proof}
[Proof of Lemma \ref{lem:algebraic sum and HLC}]
In view of (\ref{eqn:HLC X}), it suffices to prove the statement on $S^{2n+1}_1$, working with the HLC's
(\ref{eqn:HLC Pn}). Since $\mathbf{z},\,\mathbf{z}+i\,e^{i\vartheta}\,\delta\mathbf{p}\in \mathcal{C}_n$, we have
\begin{equation}
 \label{eqn:2nd order relation}
 0=\mathbf{z}^t\,\mathbf{z}+2i\,e^{i\vartheta}\,\mathbf{z}^t\,\delta\mathbf{p}-e^{2i\vartheta}\,\delta\mathbf{p}^t\,\delta\mathbf{p}=
2i\,e^{i\vartheta}\,\mathbf{z}^t\,\delta\mathbf{p}-e^{2i\vartheta}\,\|\delta\mathbf{p}\|^2,
\end{equation}
so that $i\,\mathbf{z}^t\,\delta\mathbf{p}=e^{i\vartheta}\,\|\delta\mathbf{p}\|^2/2$.

Let us look for $\beta>0$ and $\mathbf{h}\in \mathbf{z}^{\perp_h}$ (Hermitian orthocomplement) such that
\begin{equation}
 \label{eqn:tentative HLC}
 \mathbf{z}+i\,e^{i\vartheta}\,\delta\mathbf{p}=\beta\,(\mathbf{z}+\mathbf{h}).
\end{equation}
If this is possible at all, then necessarily $\beta=1/\|\mathbf{z}+\mathbf{h}\|$,
as $\left\|\mathbf{z}+i\,e^{i\vartheta}\,\delta\mathbf{p}\right\|=1$.
Then 
\begin{equation}
 \label{eqn:tentative HLC1}
 \mathbf{z}+i\,e^{i\vartheta}\,\delta\mathbf{p}=\mathbf{z}+_{S^{2n+1}_1}(0,\mathbf{h}).
\end{equation}

Assuming that (\ref{eqn:tentative HLC})
may be solved, then, taking the Hermitian product with $\mathbf{z}$ on both sides of (\ref{eqn:tentative HLC}) and using
(\ref{eqn:2nd order relation}) we get 
\begin{equation}
 \label{eqn:tentative HLC2}
\beta =\mathbf{z}^t\,\left (\overline{\mathbf{z}}-i\,e^{-i\vartheta}\,\delta\mathbf{p}\right)=
1-i\,e^{-i\vartheta}\,\mathbf{z}^t\,\delta\mathbf{p}
=1-\frac{1}{2}\,\|\delta\mathbf{p}\|^2>0.
\end{equation}

With this value of $\beta$, let us set
\begin{equation}
 \label{eqn:_defn of h}
 \mathbf{h}=:\frac{1}{\beta}\,( \mathbf{z}+i\,e^{i\theta}\,\delta\mathbf{p})-\mathbf{z},
\end{equation}
so that (\ref{eqn:tentative HLC}) is certainly satisfied. We need to verify that $\mathbf{h}\in \mathbf{z}^{\perp_h}$.
Indeed we have
$$
\mathbf{h}^t\,\overline{\mathbf{z}}=\frac{1}{\beta}\,\left( 1+i\,e^{i\theta}\delta\mathbf{p}^t\,\overline{\mathbf{z}}\right)-1
=\frac{1}{\beta}\,\left( 1-\frac{1}{2}\,\|\delta\mathbf{p}\|^2\right)-1=0.
$$
Since $\mathbf{h}=i\,e^{i\theta}\,\delta\mathbf{p}+R_2(\delta\mathbf{p})$, the proof of the Lemma is complete.
\end{proof}

Notice that $\mathbf{h}$ is given for $\delta\mathbf{p}\sim \mathbf{0}$ by an asymptotic expansion 
in homogeneous polynomials of increasing degree in $\delta\mathbf{p}$ of the form
\begin{equation}
 \label{eqn:h as asymptotic expansion}
 \mathbf{h}\sim i\,e^{i\theta}\,\delta\mathbf{p}+
 \frac{1}{2}\,\|\delta\mathbf{p}\|^2\,\mathbf{z}+\frac{i}{2}\,e^{i\theta}\,\|\delta\mathbf{p}\|^2\,\delta\mathbf{p}
 +\frac{1}{4}\,\|\delta\mathbf{p}\|^4\,\mathbf{z}+\cdots
\end{equation}
This holds on $S^{2n+1}_1$, but a similar expansion obviously holds on $X_1$, possibly with modified terms
in higher degree.

Let us apply Lemma \ref{lem:algebraic sum and HLC} with $\mathbf{z}=\mathbf{z}_0/\sqrt{2}$
and $\delta\mathbf{p}_j=e^{i\theta}\,\big(A_{jk}(\mathbf{v}_j)/\sqrt{2}\big)/\sqrt{k}$
(we'll set $\theta=0$ for $j=0$ and $\theta=\vartheta$ for $j=1$).
To this end, let us note that in view of
(\ref{eqn:defn of Ajk})
for $k\rightarrow +\infty$ there is an asymptotic expansion of the form
\begin{eqnarray}
 \label{eqn:asymtp expansion for Ak}
A_{jk}(\mathbf{v})\sim \sum_{j\ge 0}\frac{1}{k^{l/2}}\,P_{j,l+1}(\mathbf{v}), 
\end{eqnarray}
where $P_{j,l}$ is a homogeneous (vector valued) polynomial function of degree $l$, and 
$P_{j1}(\mathbf{v})=\mathbf{v}$. Hence 
\begin{equation}
 \label{eqn:Ajl}
 \delta\mathbf{p}_j=\frac{e^{i\theta}}{\sqrt{2k}}\,A_{jk}(\mathbf{v}_j)
 \sim \frac{e^{i\theta}}{\sqrt{2}}\,\sum_{j\ge 0}\frac{1}{k^{(l+1)/2}}\,P_{j,l+1}(\mathbf{v}_j)=
 \frac{e^{i\theta}}{\sqrt{k}}\,\frac{\mathbf{v}_j}{\sqrt{2}}+\cdots
\end{equation}
Making use of (\ref{eqn:Ajl}) in  (\ref{eqn:h as asymptotic expansion}) we obtain
\begin{equation}
 \label{eqn:hkj}
 \mathbf{h}_{kj}\sim \sum_{l\ge 1}\frac{1}{k^{l/2}}\,Q_{jl}(\theta;\mathbf{v})\\
 =\frac{1}{\sqrt{k}}\,\left(i\,e^{i\theta}\,\frac{\mathbf{v}_j}{\sqrt{2}}
 +\sum_{l\ge 1}\frac{1}{k^{l/2}}\,Q_{j,l+1}(\theta;\mathbf{v}_j)\right)
\end{equation}
where $Q_{jl}(\theta;\cdot)$ is a homogeneous polynomial function of degree $l$,
and we have emphasized the dependence on $k$.

Thus we obtain for $j=0$ (with $\theta=0$) that
\begin{equation}
 \label{eqn:da traslato a HLC}
 \frac{\mathbf{z}_0}{\sqrt{2}}+\frac{1}{\sqrt{k}}\,
\frac{i\,\,A_{0k}(\mathbf{v}_0)}{\sqrt{2}}=
\frac{\mathbf{z}_0}{\sqrt{2}}+_{X_1}(0,\mathbf{h}_{k0}),
\end{equation}
where
\begin{equation}
 \label{eqn:espansione per h}
 \mathbf{h}_{k0}\sim \frac{1}{\sqrt{k}}\,\left(i\,\frac{\mathbf{v}_0}{\sqrt{2}}
 +\sum_{l\ge 1}\frac{1}{k^{l/2}}\,Q_{0,l+1}(0;\mathbf{v}_0)\right)=\frac{1}{\sqrt{k}}\,\mathbf{a}_{k0},
\end{equation}
with $\mathbf{a}_{k0}$ defined by the latter equality.
Similarly, for $j=1$ (with $\theta=\vartheta$) we have
$$
\frac{\mathbf{z}_0}{\sqrt{2}}+\frac{1}{\sqrt{k}}\,\frac{i\,e^{i\vartheta}\,\,A_{1k}(\mathbf{v}_1)}{\sqrt{2}}=
\frac{\mathbf{z}_0}{\sqrt{2}}+_{X_1}(0,\mathbf{h}_{k1}),
$$
where
$$
\mathbf{h}_{k1}\sim \frac{1}{\sqrt{k}}\,\left(i\,e^{i\vartheta}\,\frac{\mathbf{v}_1}{\sqrt{2}}
 +\sum_{l\ge 1}\frac{1}{k^{l/2}}\,Q_{1,l+1}(\vartheta;\mathbf{v}_1)\right)=\frac{1}{\sqrt{k}}\,\mathbf{a}_{k1}.
$$

Let us return to (\ref{eqn:szego term integrand}). In view of Theorem
\ref{thm:sz expansion szego equiv}, we get
\begin{eqnarray}
 \label{eqn:szego term expanded expanded}
 \lefteqn{\Pi_{1,k}\left(\frac{\mathbf{z}_0}{\sqrt{2}}+\frac{1}{\sqrt{k}}\,
\frac{i\,\,A_{0k}(\mathbf{v}_0)}{\sqrt{2}},
\frac{\mathbf{z}_0}{\sqrt{2}}+\frac{1}{\sqrt{k}}\,\frac{i\,e^{i\vartheta}\,\,A_{1k}(\mathbf{v}_1)}{\sqrt{2}}\right)}\\
&=&\Pi_{1,k}\left(\frac{\mathbf{z}_0}{\sqrt{2}}+_{X_1}\frac{1}{\sqrt{k}}\,\mathbf{a}_{k0},
\frac{\mathbf{z}_0}{\sqrt{2}}+_{X_1}\frac{1}{\sqrt{k}}\,\mathbf{a}_{k1}\right)\nonumber\\
&\sim&\left(\frac{k}{\pi}\right)^{n-1}\,e^{\psi_2(\mathbf{a}_{k0},\mathbf{a}_{k1})}\cdot
\left[1+\sum_{b=1}^{+\infty}k^{-b/2}\,P_b(\mathbf{a}_{k0},\mathbf{a}_{k1})\right].\nonumber
\end{eqnarray}

We have
\begin{equation}
 \label{eqn:psi2 espanso}
 \psi_2(\mathbf{a}_{k0},\mathbf{a}_{k1})\sim\frac{1}{2}\,\psi_2\left(\mathbf{v}_0,e^{i\vartheta}\,\mathbf{v}_1\right)+
 \sum_{l\ge 1}\frac{1}{k^{l/2}}\,\widetilde{Q}_{l+2}(\vartheta;\mathbf{v}_0,\mathbf{v}_1),
\end{equation}
where $\widetilde{Q}_{l}(\vartheta;\cdot,\cdot)$ is a homogeneous $\mathbb{C}$-valued polynomial of degree $l$.
For any $r\ge 1$ and $l_1,\ldots,l_r\ge 1$, we have 
\begin{eqnarray*}
 \prod_{j=1}^r\frac{1}{k^{l_j/2}}\,\widetilde{Q}_{l_j+2}(\vartheta;\mathbf{v}_0,\mathbf{v}_1)
 &=&\frac{1}{k^{\sum_{j=1}^rl_j/2}}\,\widehat{Q}_{\sum_{j=1}^rl_j+2r}(\vartheta;\mathbf{v}_0,\mathbf{v}_1),
\end{eqnarray*}
where $\widehat{Q}_{\sum_{j=1}^rl_j+2r}(\vartheta;\cdot,\cdot)$ is homogeneous of degree $\sum_{j=1}^rl_j+2r$.
Since $l_j\ge 1$ for every $j$, we have $\sum_{j=1}^rl_j+2r\le 3\,\sum_{j=1}^rl_j$.

One can see from this that 
\begin{equation}
 \label{eqn:szego exponential}
 e^{\psi_2(\mathbf{a}_{k0},\mathbf{a}_{k1})}\sim 
 e^{\frac{1}{2}\,\psi_2\left(\mathbf{v}_0,e^{i\vartheta}\,\mathbf{v}_1\right)}\,
 \left[1+\sum_{l\ge 1}\frac{1}{k^{l/2}}\,B_l(\vartheta;\mathbf{v}_0,\mathbf{v}_1)\right],
\end{equation}
where $B_l(\vartheta;\cdot,\cdot)$ is a polynomial of degree $\le 3l$, and having the same parity as $l$.

Similarly, recalling that $P_b$ has the same parity as $b$ and degree $\le 3b$,  
each summand $k^{-b/2}\,P_b(\mathbf{a}_{k0},\mathbf{a}_{k1})$ in (\ref{eqn:szego term expanded expanded}) gives rise to an 
asymptotic expansion in terms of the form 
\begin{eqnarray*}
 \frac{1}{k^{b/2}}\,\prod_{a=1}^r\frac{1}{k^{l_a/2}}\,R_{l_a+1}(\vartheta;\mathbf{v}_0,\mathbf{v}_1)=
 \frac{1}{k^{(b+\sum_{a=1}^rl_a)/2}}\,\widetilde{R}_{\sum_{a=1}^rl_a+r}(\vartheta;\mathbf{v}_0,\mathbf{v}_1),
\end{eqnarray*}
where $R_l(\vartheta;\cdot,\cdot)$ and $\widetilde{R}_l(\vartheta;\cdot,\cdot)$ are homogeneous
polynomials of the given degree, $r\le 3b$, and $b-r$ is even.
Then $3\,(b+\sum_{a=1}^rl_a)\ge \sum_{a=1}^rl_a+r$, and $(b+\sum_{a=1}^rl_a)-(\sum_{a=1}^rl_a+r)=b-r$
is also even. Hence each summand $k^{-b/2}\,P_b(\mathbf{a}_{k0},\mathbf{a}_{k1})$ ($b\ge 1$) yields an asymptotic expansion 
of the form 
$$
\sum_{l\ge 1}\,k^{-l/2}\,T_{bl}(\vartheta;\mathbf{v}_0,\mathbf{v}_1),
$$
where again each $T_{bl}$ has the same parity as $l$ and degree $\le 3l$.

Putting this all together, we obtain an asymptotic expansion for the integrand in 
(\ref{eqn:push forward explicit k+rescaled}):

\begin{lem}
 \label{lem:asymptotic expansion for Piak}
For $l\ge 0$, there exist polynomials $Z_l(\vartheta;\cdot,\cdot)$ of degree $\le 3l$ and parity $(-1)^l$,
with $Z_0(\vartheta;\cdot,\cdot)=1$,
such that 
 \begin{eqnarray*}
 \lefteqn{\Pi_{\sqrt{2},k}\left(\mathbf{z}_0+\frac{i}{\sqrt{k}}\,A_{0k}(\mathbf{v}_0),
 \mathbf{z}_1+\frac{i}{\sqrt{k}}\,A_{1k}(\mathbf{v}_1)\right)
\cdot \mathcal{V}\left(\dfrac{\mathbf{v}_0}{\sqrt{k}},\frac{\mathbf{v}_1}{\sqrt{k}}\right)}\\
&\sim&\frac{\sqrt{2}}{2^n}\,e^{ik\vartheta}\,\left(\frac{k}{\pi}\right)^{n-1}\,e^{\frac{1}{2}\,
\psi_2\left(\mathbf{v}_0,e^{i\vartheta}\,\mathbf{v}_1\right)}\,
 \sum_{l\ge 0}\frac{1}{k^{l/2}}\,Z_l(\vartheta;\mathbf{v}_0,\mathbf{v}_1).
 \end{eqnarray*}
\end{lem}

\begin{proof}
[Proof of Lemma \ref{lem:asymptotic expansion for Piak}]
The previous arguments yield an asymptotic expansion of the given form for the first factor. 
We need only multiply the
latter expansion by the Taylor expansion of the second factor. 
\end{proof}

Since integration in (\ref{eqn:push forward explicit k+rescaled}) takes place over a poly-ball or radius
$O\left(k^{\epsilon}\right)$ in $\left(\mathbf{q}_0^\perp\cap \mathbf{q}_1^\perp\right)^2$, the expansion may be integrated
term by term. In addition, given that the exponent and the cut-offs are even functions of 
$(\mathbf{v}_0,\mathbf{v}_1)$, only terms of even parity yield a non-zero integral. Hence we may discard the half-integer powers
and obtain
\begin{eqnarray}
 \label{eqn:push forward explicit k+rescaled expanded}
 \mathcal{P}_k(\mathbf{q}_0,\mathbf{q}_1)_{+}
 &\sim&\frac{k^{1-n}}{C_{k,n}^2}\,\frac{\sqrt{2}}{2^n}\,e^{ik\vartheta}\,\left(\frac{k}{\pi}\right)^{n-1}\,
 \sum_{l\ge 0}k^{-l}\,\widehat{P}_{l}(\vartheta)_{+},
\end{eqnarray}
where
\begin{eqnarray}
\label{eqn:lth term P_l}
\widehat{P}_{l}(\vartheta)_{+}&=:&\int_{\mathbf{q}_0^\perp\cap \mathbf{q}_1^\perp}\,
 \int _{\mathbf{q}_0^\perp\cap \mathbf{q}_1^\perp}\,
 \varrho\left(k^{-\epsilon}\,\mathbf{v}_0\right)\,\varrho\left(k^{-\epsilon}\,\mathbf{v}_1\right)\nonumber\\
&&\cdot e^{\frac{1}{2}\,
\psi_2\left(\mathbf{v}_0,e^{i\vartheta}\,\mathbf{v}_1\right)}\,
 Z_{2l}(\vartheta;\mathbf{v}_0,\mathbf{v}_1)\,\mathrm{d}\mathbf{v}_0\,
 \mathrm{d}\mathbf{v}_1.
\end{eqnarray}
We can slightly simplify the previous asymptotic expansion, as follows. 
First, as emphasized the dependence on $(\mathbf{q}_0,\mathbf{q}_1)$ is of course only through the angle $\vartheta$.
In particular, in (\ref{eqn:lth term P_l}) nothing is lost by assuming that
$\mathbf{q}_0$ and $\mathbf{q}_1$ span the 2-plane $\{\mathbf{0}\}\times \mathbb{R}^2\subseteq \mathbb{R}^{n+1}$,
and therefore that $\mathbf{q}_0^\perp\cap \mathbf{q}_1^\perp=\mathbb{R}^{n-1}\times \{\mathbf{0}\}$.

Furthermore, given (\ref{eqn:defn di psi2}), we have
\begin{eqnarray}
 \label{eqn:psi2 sviluppato}
\lefteqn{\psi_2\left(\mathbf{v}_0,e^{i\vartheta}\,\mathbf{v}_1\right)}\nonumber\\
&=&-i\,\sin(\vartheta)\,\mathbf{v_0}^t\,\mathbf{v}_1-\frac{1}{2}\,\|\mathbf{v_0}-\cos(\vartheta)\,\mathbf{v}_1\|^2
-\frac{1}{2}\,\sin(\vartheta)^2\,\|\mathbf{v}_1\|^2.
\end{eqnarray}
With the change of variables
$$
\begin{pmatrix}
 \mathbf{v}_0\\
 \mathbf{v_1}
\end{pmatrix}=\sqrt{2}\,
\begin{pmatrix}
 \mathbf{b}_0+\cot(\vartheta)\,\mathbf{b}_1\\
\big(1/\sin(\vartheta)\big)\,\mathbf{b}_1
\end{pmatrix}
$$
we obtain
\begin{eqnarray}
 \label{eqn:psi2 sviluppato nuove variabili}
\psi_2\left(\mathbf{v}_0,e^{i\vartheta}\,\mathbf{v}_1\right)
=-\frac{1}{2}\,\|\mathbf{b}_0\|^2-i\,\mathbf{b}_0^t\,\mathbf{b}_1
-\frac{1}{2}\,\big(1+2i\,\cot(\vartheta)\big)\,\|\mathbf{b_1}\|^2.
\end{eqnarray}
Since $Z_{2l}(\vartheta,\cdot,\cdot)$ is even and has degree $\le 6l$, we can write
$$ 
Z_{2l}\left(\vartheta;\mathbf{b}_0+\cot(\vartheta)\,\mathbf{b}_1,\frac{\mathbf{b}_1}{\sin(\vartheta)}\right)=
\frac{1}{\sin(\vartheta)^{6l}}\,T_{l}(\vartheta;\mathbf{b}_0,\mathbf{b}_1),
$$
where $T_l(\vartheta;\cdot,\cdot)$ is an even polynomial of degree $\le 6l$, with smooth bounded coefficients for $\vartheta\in [0,\pi]$.
Thus
\begin{eqnarray}
 \label{eqn:computation of P0}
\lefteqn{\widehat{P}_{l}(\vartheta)_{+}}\nonumber\\
&=&\left(\frac{2}{\sin(\vartheta)}\right)^{n-1}\,\frac{1}{\sin(\vartheta)^{6l}}\,\int_{\mathbb{R}^{n-1}}\,\int_{\mathbb{R}^{n-1}}\,
e^{-\frac{1}{2}\,\|\mathbf{b}_0\|^2-i\,\mathbf{b}_0^t\,\mathbf{b}_1
-\frac{1}{2}\,\big(1+2i\,\cot(\vartheta)\big)\,\|\mathbf{b_1}\|^2}\nonumber\\
&&\cdot \varrho\left(k^{-\epsilon}\,\sqrt{2}\,\big(\mathbf{b}_0+\cot(\vartheta)\,\mathbf{b}_1\big)\right)\,\varrho\left(k^{-\epsilon}\,\sqrt{2}\,
\sin(\vartheta)^{-1}\,\mathbf{b}_1\right)\,T_{l}(\vartheta;\mathbf{b}_0,\mathbf{b}_1)\nonumber\\
&&\cdot 
\mathrm{d}\mathbf{b}_0\,\mathrm{d}\mathbf{b}_1.
\end{eqnarray}

There is a constant $C>0$ such that the support of 
$$
1-\varrho\left(k^{-\epsilon}\,\sqrt{2}\,\big(\mathbf{b}_0+\cot(\vartheta)\,\mathbf{b}_1\big)\right)\,\varrho\left(k^{-\epsilon}\,\sqrt{2}\,
\sin(\vartheta)^{-1}\,\mathbf{b}_1\right)
$$
is contained in the locus where $\|(\mathbf{b}_0,\mathbf{b}_1)\|\ge C\,k^{\epsilon}\,\sin(\vartheta)$. Under the assumptions of the Theorem, 
this implies, perhaps for a different constant $C>0$, that
$\|(\mathbf{b}_0,\mathbf{b}_1)\|\ge C\,k^{\epsilon-\delta}$. On the other hand, the exponent in (\ref{eqn:computation of P0})
satisfies
\begin{equation*}
 \left|-\frac{1}{2}\,\|\mathbf{b}_0\|^2-i\,\mathbf{b}_0^t\,\mathbf{b}_1
-\frac{1}{2}\,\big(1+2i\,\cot(\vartheta)\big)\,\|\mathbf{b_1}\|^2\right|\le
-\frac{1}{2}\,\left(\|\mathbf{b}_0\|^2+\|\mathbf{b}_1\|^2\right).
\end{equation*}
Given that $\epsilon>\delta$ (statement of Proposition \ref{prop:shrinking delta}), 
we conclude that only a negligible contribution to the asymptotics is lost, if the cut-off function is omitted and integration
is now extended to all of $\mathbb{R}^{n-1}\times \mathbb{R}^{n-1}$.

We can thus rewrite (\ref{eqn:push forward explicit k+rescaled expanded}) as follows:
\begin{eqnarray}
 \label{eqn:push forward explicit k+rescaled expanded1}
 \lefteqn{
 \mathcal{P}_k(\mathbf{q}_0,\mathbf{q}_1)_{+}}\nonumber\\
 &\sim&\frac{1}{C_{k,n}^2}\,\frac{\sqrt{2}}{2}\,e^{ik\vartheta}\,\left(\frac{1}{\pi\,\sin(\vartheta)}\right)^{n-1}\,
 \sum_{l\ge 0}k^{-l}\,\frac{1}{\sin(\vartheta)^{6l}}\,\widetilde{P}_{l}(\vartheta)_{+},
\end{eqnarray}
where
\begin{eqnarray}
 \label{eqn:lth coefficient simple}
\lefteqn{\widetilde{P}_{l}(\vartheta)_{+}}\nonumber\\
&=&\int_{\mathbb{R}^{n-1}}\,\int_{\mathbb{R}^{n-1}}\,
e^{-\frac{1}{2}\,\|\mathbf{b}_0\|^2-i\,\mathbf{b}_0^t\,\mathbf{b}_1
-\frac{1}{2}\,\big(1+2i\,\cot(\vartheta)\big)\,\|\mathbf{b_1}\|^2}\, T_{l}(\vartheta;\mathbf{b}_0,\mathbf{b}_1)\,
\mathrm{d}\mathbf{b}_0\,\mathrm{d}\mathbf{b}_1.\nonumber
\end{eqnarray}

Let us set $B_\vartheta=\big(1+i\,\cot(\vartheta)\big)\,I_{n-1}$.
The leading order coefficient is
\begin{eqnarray}
 \label{eqn:leading order coefficient}
\lefteqn{\widetilde{P}_{0}(\vartheta)_{+}}\nonumber\\
&=&\int_{\mathbb{R}^{n-1}}e^{-\frac{1}{2}\,\big(1+2i\,\cot(\vartheta)\big)\,\|\mathbf{b_1}\|^2}\,\left[\int_{\mathbb{R}^{n-1}}\,
e^{-i\,\mathbf{b}_0^t\,\mathbf{b}_1
-\frac{1}{2}\,\|\mathbf{b}_0\|^2}\,\mathrm{d}\mathbf{b}_0\right]\, 
\mathrm{d}\mathbf{b}_1\nonumber\\
&=&(2\pi)^{(n-1)/2}\,\int_{\mathbb{R}^{n-1}}e^{-\frac{1}{2}\,\big(2+2i\,\cot(\vartheta)\big)\,\|\mathbf{b_1}\|^2}\,
\mathrm{d}\mathbf{b}_1\nonumber\\
&=&2^{(n-1)/2}\,\pi^{n-1}\,\frac{1}{\sqrt{\det(B_\vartheta)}}\nonumber\\
&=&\left(\sqrt{2}\,\pi\right)^{n-1}\,
\sin(\vartheta)^{(n-1)/2}\,e^{i\,\left(\frac{\vartheta}{2}-\frac{\pi}{4}\right)\,(n-1)}.
\end{eqnarray}

Given (\ref{eqn:leading order coefficient}), (\ref{eqn:push forward explicit k+rescaled expanded1}) and
(\ref{eqn:push forward real part}), $\mathcal{P}_k(\mathbf{q}_0,\mathbf{q}_1)$ has an asymptotic
expansion for $k\rightarrow +\infty$ with leading order term
\begin{equation}
 \label{eqn:leading term Pk}
\frac{2^{n/2}}{C_{k,n}^2}\,\frac{1}{\sin(\vartheta)^{(n-1)/2}}\,\cos\left(k\,\vartheta+
\left(\frac{\vartheta}{2}-\frac{\pi}{2}\right)\,(n-1)\right).
\end{equation}

For any $l$, we can write
\begin{eqnarray}
 \label{eqn:lth coefficient simple integrated}
\lefteqn{\widetilde{P}_{l}(\vartheta)_{+}}\nonumber\\
&=&\int_{\mathbb{R}^{n-1}}e^{-\frac{1}{2}\,\big(1+2i\,\cot(\vartheta)\big)\,\|\mathbf{b_1}\|^2}\,\left[\int_{\mathbb{R}^{n-1}}\,
e^{-i\,\mathbf{b}_0^t\,\mathbf{b}_1
-\frac{1}{2}\,\|\mathbf{b}_0\|^2}\, T_{l}(\vartheta;\mathbf{b}_0,\mathbf{b}_1)\,\mathrm{d}\mathbf{b}_0\right]\, 
\mathrm{d}\mathbf{b}_1\nonumber\\
&=&\int_{\mathbb{R}^{n-1}}e^{-\frac{1}{2}\,\big(2+2i\,\cot(\vartheta)\big)\,\|\mathbf{b_1}\|^2}\,
\mathcal{T}_{l}(\vartheta;\mathbf{b}_1)\,\mathrm{d}\mathbf{b}_1,
\end{eqnarray}
where $\mathcal{T}_{l}(\vartheta;\cdot)$ is an even polynomial of degree $\le 6l$.

Let us introduce the 
Fourier transform
\begin{eqnarray}
 \label{eqn:fourier trasnform}
 \mathcal{F}(\mathbf{c})&=& \int_{\mathbb{R}^{n-1}}e^{-\frac{1}{2}\,\big(2+2i\,\cot(\vartheta)\big)\,\|\mathbf{b_1}\|^2-i\,\mathbf{b}_1^t\,\mathbf{c}}\,
 \mathrm{d}\mathbf{b}_1\\
 &=&(2\pi)^{(n-1)/2}\,\sin(\vartheta)^{(n-1)/2}\,e^{i\,\left(\frac{\vartheta}{2}-\frac{\pi}{4}\right)\,(n-1)}\,
 e^{-\frac{1}{2}\,\big(2+2i\,\cot(\vartheta)\big)^{-1}\,\|\mathbf{c}\|^2}.\nonumber
\end{eqnarray}

Then (\ref{eqn:lth coefficient simple integrated}) is the result of applying an 
even differential polynomial $P_l(D_\mathbf{c})$ of degree $\le 6l$ to $ \mathcal{F}(\mathbf{c})$,
and then evaluating the result at $\mathbf{c}=\mathbf{0}$.

Given this and (\ref{eqn:leading term Pk}), we conclude that 
\begin{eqnarray}
 \label{eqn:finale expansion}
 \lefteqn{\mathcal{P}_k(\mathbf{q}_0,\mathbf{q}_1)=\frac{2^{\frac{n}{2}}}{C_{k,n}^2}\,
\left(\frac{1}{\sin(\vartheta)}\right)^{(n-1)/2}\,}\\
 &&\cdot\left[\cos\left(k\vartheta+\left(\frac{\vartheta}{2}-\frac{\pi}{4}\right)(n-1)\right)\cdot A(\vartheta)
+\sin\left(k\vartheta+\left(\frac{\vartheta}{2}-\frac{\pi}{4}\right)(n-1)\right)\cdot B(\vartheta)\right]\nonumber,
\end{eqnarray}
where 
$$
A(\vartheta)\sim 1+\sum_{l=1}^{+\infty}k^{-l}\,\frac{A_l(\vartheta)}{\sin(\vartheta)^{6l}},
\,\,\,\,\,\,B(\vartheta)\sim \sum_{l=1}^{+\infty}k^{-l}\,\frac{B_l(\vartheta)}{\sin(\vartheta)^{6l}},
$$
with $A_l$ and $B_l$ smooth functions of $\vartheta$ on $[0,2\pi]$.

\end{proof}

\section{Proof of Proposition \ref{prop:espansione per Cnk}}

\begin{proof}
 [Proof of Proposition \ref{prop:espansione per Cnk}]
The diagonal restriction $\mathcal{P}_{k,n}(\mathbf{q},\mathbf{q})$
may be computed in two different ways. On the one hand, 
since $\mathcal{P}_{k,n}(\mathbf{q},\mathbf{q})$ is constant we have
\begin{equation}
 \label{eqn:Pk riemann roch}
 \mathcal{P}_{k,n}(\mathbf{q},\mathbf{q})=\frac{N_{k,n}}{\mathrm{vol}(S^n)}=
 \frac{2}{\mathrm{vol}(S^n)}\,\frac{k^{n-1}}{(n-1)!}+O\left(k^{n-2}\right).
\end{equation}

On the other hand, (\ref{eqn:push forward explicit}) with $\mathbf{q}_0=\mathbf{q}_1=\mathbf{q}$ yields
\begin{eqnarray}
 \label{eqn:push forward explicit diagonal}
 \lefteqn{\mathcal{P}_{k,n}(\mathbf{q},\mathbf{q})}\nonumber\\
 &=&\frac{1}{C_{k,n}^2}\,\int_{S(\mathbf{q}^\perp)}\,\int _{S({\mathbf{q}}^\perp)}\,
 \Pi_{\sqrt{2},k}(\mathbf{q}+i\,\mathbf{p},\mathbf{q}+i\,\mathbf{p}')\,\mathrm{d}V_{S(\mathbf{q}^\perp)}(\mathbf{p})\,
 \mathrm{d}V_{S({\mathbf{q}}^\perp)}(\mathbf{p}') \nonumber\\
 &=&\frac{1}{C_{k,n}^2}\,\int_{S(\mathbf{q}^\perp)}\,F_k(\mathbf{q},\mathbf{p})\,\mathrm{d}V_{S(\mathbf{q}^\perp)}(\mathbf{p}),
\end{eqnarray}
where
\begin{equation}
 \label{eqn:Fk}
 F_k(\mathbf{q},\mathbf{p})=:\int _{S({\mathbf{q}}^\perp)}\,
 \Pi_{\sqrt{2},k}(\mathbf{q}+i\,\mathbf{p},\mathbf{q}+i\,\mathbf{p}')\,\mathrm{d}V_{S(\mathbf{q}^\perp)}(\mathbf{p}').
\end{equation}

Again, integration in $\mathrm{d}V_{S(\mathbf{q}^\perp)}(\mathbf{p}')$
localizes in a shrinking neighborhood of $\mathbf{p}$. Hence we may let
$$
\mathbf{p}'=\mathbf{p}+A(\mathbf{v}),\,\,A(\mathbf{v})=\mathbf{v}+S_+(\mathbf{v})\,\mathbf{p},
$$
where $\mathbf{v}\in \mathbf{q}^\perp\cap\mathbf{p}^\perp$, and introduce the cut-off
$\varrho\left(k^{1/2-\epsilon}\,\mathbf{v}\right)$. Passing to rescaled coordinates, and setting 
$\mathbf{z}=\mathbf{q}+i\,\mathbf{p}$,
we get 
\begin{eqnarray}
\label{eqn:Fkespanso}
 \lefteqn{F_k(\mathbf{q},\mathbf{p})}\\
 &=&\frac{1}{k^{(n-1)/2}}\,\int_{\mathbf{q}^\perp\cap\mathbf{p}^\perp}
 \,\varrho\left(k^{-\epsilon}\,\mathbf{v}\right)\,
 \Pi_{\sqrt{2},k}\left(\mathbf{z},\mathbf{z}+\frac{i}{\sqrt{k}}\,A_k(\mathbf{v})\right)\,
\mathcal{V}\left(\frac{\mathbf{v}}{\sqrt{k}}\right) \,\mathrm{d}\mathbf{v}.\nonumber
\end{eqnarray}
where
$$
A_k(\mathbf{v})=\mathbf{v}+\sqrt{k}\,S_+\left(\frac{\mathbf{v}}{\sqrt{k}}\right)\,\mathbf{p},\,\,\,\,\,\,
\mathcal{V}(\mathbf{0})=1.
$$

By Lemma \ref{lem:asymptotic expansion for Piak} (with $\mathbf{z}=\mathbf{z}_0=\mathbf{z}_1$,
$\mathbf{v}_0=\mathbf{0}$, $\mathbf{v}_1=\mathbf{v}$, $\vartheta=0$), we have 
\begin{eqnarray*}
 \lefteqn{\Pi_{\sqrt{2},k}\left(\mathbf{z},
 \mathbf{z}+\frac{i}{\sqrt{k}}\,A_{k}(\mathbf{v})\right)
\cdot \mathcal{V}\left(\frac{\mathbf{v}}{\sqrt{k}}\right)}\\
&\sim&\frac{\sqrt{2}}{2^n}\,\left(\frac{k}{\pi}\right)^{n-1}\,e^{-\frac{1}{4}\,
\|\mathbf{v}\|^2}\,
 \sum_{l\ge 0}\frac{1}{k^{l/2}}\,Z_l(\mathbf{v}),
 \end{eqnarray*}
for certain polynomials $Z_l$ of degree $\le 3l$ and parity $(-1)^l$,
with $Z_0(\cdot)=1$.

As before, the expansion may be integrated term by term and, by parity, only the summands with 
$l$ even yield a non-zero contribution. In addition, only a negligible contribution is lost if
the cut off is omitted and 
integration is extended to all of $\mathbf{q}^\perp\cap\mathbf{p}^\perp\cong \mathbb{R}^{n-1}$.
Therefore 
\begin{eqnarray}
\label{eqn:Fkespanso1}
 F_k(\mathbf{q},\mathbf{p})
 &\sim&\frac{\sqrt{2}}{2^n}\,\frac{k^{(n-1)/2}}{\pi^{n-1}}\,\sum_{l\ge 0}k^{-l}\int_{\mathbb{R}^{n-1}}
 \,e^{-\frac{1}{4}\,
\|\mathbf{v}\|^2}\,Z_{2l}(\mathbf{v})\,\mathrm{d}\mathbf{v}\nonumber\\
&=&\frac{1}{\sqrt{2}}\,\left(\frac{k}{\pi}\right)^{(n-1)/2}+\cdots
\end{eqnarray}

Inserting this in (\ref{eqn:push forward explicit diagonal}), we obtain
an asymptotic expansion 
\begin{equation}
\label{eqn:push forward explicit diagonal1}
 \mathcal{P}_{k,n}(\mathbf{q},\mathbf{q})\sim
 \frac{\mathrm{vol}(S^{n-1})}{C_{k,n}^2}\,\frac{1}{\sqrt{2}}\,\left(\frac{k}{\pi}\right)^{(n-1)/2}+\cdots
 \end{equation}
 
 Comparing (\ref{eqn:Pk riemann roch}) and
 (\ref{eqn:push forward explicit diagonal1}), we obtain an asymptotic expansion in descending powers of $k$,
 of the form
 $$
 C_{k,n}\sim \left[\frac{\mathrm{vol}(S^n)\,\mathrm{vol}(S^{n-1})}{2\,\sqrt{2}}\cdot (n-1)!\right]^{1/2}\,
 (\pi\,k)^{-(n-1)/4}+\cdots$$
\end{proof}


\begin{thebibliography}{Dillo99}

\bibitem[AH]{ah} K. Atkinson, W. Han, {\em
Spherical harmonics and approximations on the unit sphere: an introduction}, Lecture Notes in Mathematics \textbf{2044}, Springer, Heidelberg, 2012

\bibitem[B]{b} H. Bateman, {\em Higher Transcendental Functions Volume II}, (Bateman Manuscript Project) Mc Graw-Hill Book Company, 1953

\bibitem[BSZ]{bsz} P. Bleher, B. Shiffman, S. Zelditch, {\em
Universality and scaling of correlations between zeros on complex
manifolds}, Invent. Math. \textbf{142} (2000), 351--395






\bibitem[C]{c} M. Christ, 
{\em Slow off-diagonal decay for Szeg\"{o} kernels associated to smooth Hermitian line bundles}, 
Harmonic analysis at Mount Holyoke (South Hadley, MA, 2001), 77–89

\bibitem[G]{g} V. Guillemin, {\em Toeplitz operators in n dimensions}, Integral Equations and Operator Theory \textbf{7} (1984), no. 2, 145–205

\bibitem[L]{l}  G. Lebeau, {\em Fonctions harmoniques et spectre singulier}, Ann. Sci. \'{E}cole Norm. Sup. (4) \textbf{13} (1980), no. 2, 269–291

\bibitem[Leb]{leb}  N. N. Lebedev, {\em Special functions and their applications}, 
Revised English edition. Translated and edited by Richard A. Silverman, Prentice-Hall, Inc., Englewood Cliffs, N.J. 1965 

\bibitem[M]{m0} C. M\"{u}ller, 
{\em Spherical harmonics}, Lecture Notes in Mathematics, \textbf{17} Springer-Verlag, Berlin-New York 1966

\bibitem[M1]{m1} C. M\"{u}ller, 
{\em Analysis of spherical symmetries in Euclidean spaces}, Applied Mathematical Sciences, \textbf{129}, Springer-Verlag, New York, 1998


\bibitem[O]{o} F. Olver, {\em Asymptotics and special functions}, 
Reprint of the 1974 original [Academic Press, New York]. AKP Classics. A K Peters, Ltd., Wellesley, MA, 1997

\bibitem[SZ]{sz} B. Shiffman, S. Zelditch, {\em Asymptotics of almost
holomorphic sections of ample line bundles on symplectic
manifolds}, J. Reine Angew. Math. {\bf 544} (2002), 181--222


\bibitem[S]{s} G. Szeg\"{o}, {\em Orthogonal polynomials}, Third edition. American Mathematical Society Colloquium Publications, Vol. \textbf{23}, 
American Mathematical Society, Providence, R.I., 1967
\end{thebibliography}
\end{document}